\documentclass[a4paper]{amsart}
\usepackage{amsthm,amsfonts,amsmath,amssymb}
\usepackage[abs]{overpic}

\newtheorem{theorem}{Theorem}[section]
\newtheorem{lemma}[theorem]{Lemma}

\newtheorem{proposition}[theorem]{Proposition}

\theoremstyle{definition}

\theoremstyle{remark}

\newcommand{\e}{\varepsilon}

\newcommand{\Z}{\mathbb{Z}}

\makeatletter

\@addtoreset{figure}{section}
\makeatother

\makeatletter
  
  \@addtoreset{equation}{section}
\makeatother

\begin{document}
\title[Virtualized Delta, sharp, and pass moves]{Virtualized Delta, sharp, and pass moves for oriented virtual knots and links}

\author{Takuji NAKAMURA}
\address{Faculty of Education, 
University of Yamanashi,
Takeda 4-4-37, Kofu, Yamanashi, 400-8510, Japan}
\email{takunakamura@yamanashi.ac.jp}

\author{Yasutaka NAKANISHI}
\address{Department of Mathematics, Kobe University, 
Rokkodai-cho 1-1, Nada-ku, Kobe 657-8501, Japan}
\email{nakanisi@math.kobe-u.ac.jp}

\author{Shin SATOH}
\address{Department of Mathematics, Kobe University, 
Rokkodai-cho 1-1, Nada-ku, Kobe 657-8501, Japan}
\email{shin@math.kobe-u.ac.jp}

\author[Kodai Wada]{Kodai Wada}
\address{Department of Mathematics, Kobe University, Rokkodai-cho 1-1, Nada-ku, Kobe 657-8501, Japan}
\email{wada@math.kobe-u.ac.jp}

\makeatletter
\@namedef{subjclassname@2020}{%
  \textup{2020} Mathematics Subject Classification}
\makeatother


\thanks{This work was supported by JSPS KAKENHI Grant Numbers 
JP20K03621, JP19K03492, JP22K03287, and JP23K12973.}



\begin{abstract}
We study virtualized Delta, sharp, and pass moves 
for oriented virtual links, 
and give necessary and sufficient conditions 
for two oriented virtual links to be related by 
the local moves. 
In particular, they are unknotting operations 
for oriented virtual knots. 
We provide lower bounds for the unknotting numbers 
and prove that they are best possible. 
\end{abstract}

\maketitle

\section{Introduction} 
A local move is one of the main tools in classical knot theory 
which studies a relationship between 
topological and algebraic structures of classical knots and links in the $3$-sphere. 
For example, 
the $\Delta$-move corresponds to the set of linking numbers of classical links; 
that is, two classical links are related by a finite sequence of $\Delta$-moves 
if and only if they have the same pairwise linking numbers. 
In particular, 
the $\Delta$-move is an unknotting operation for classical knots. 

On the other hand, it is known that the $\Delta$-move is 
not an unknotting operation for virtual knots (cf.~\cite{ST}). 
In our previous paper~\cite{NNSW}, 
we introduced a more elemental move 
called a virtualized $\Delta$-move (or a $v\Delta$-move simply) 
for \textit{unoriented} virtual knots and links 
such that an ordinal $\Delta$-move is decomposed into 
a pair of virtualized $\Delta$-moves. 
See Figure~\ref{vDelta}.  
It has been shown in~\cite{NNSW} that the virtualized $\Delta$-move 
is an unknotting operation for unoriented virtual knots, 
and corresponds to the set of invariants called the parities 
for unoriented virtual links. 

\begin{figure}[htbp]
  \centering
    \begin{overpic}[width=8cm]{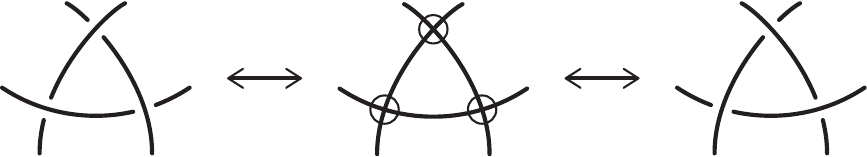}
      \put(63,26){$v\Delta$}
      \put(152,26){$v\Delta$}
    \end{overpic}
  \caption{A virtualized $\Delta$-move for an oriented virtual knot or link}
  \label{vDelta}
\end{figure}

In this paper, we study virtualized $\Delta$-moves 
for \textit{oriented} virtual knots and links, 
which are divided into two classes called 
$v\Delta^{\wedge}$-moves and $v\Delta^{\circ}$-moves 
according to the orientations of the strings involved in the moves. 
Furthermore we introduce a virtualized $\sharp$-move 
(or a $v\sharp$-move) and a virtualized pass-move 
(or a $vp$-move) as elemental versions of 
an ordinal $\sharp$-move and an ordinal pass-move, respectively.

For $X\in\{v\Delta, v\Delta^{\wedge}, v\Delta^{\circ}, v\sharp, vp\}$, 
we say that two oriented virtual links are \textit{$X$-equivalent} 
if they are related by a finite sequence of $X$-moves. 
Then we will prove the following 
by using the $i$th parity $p_i(L)\in{\Z}/2{\Z}$ and 
$i$th intersection number $\lambda_i(L)\in{\Z}$ 
of an oriented $n$-component virtual link $L$ $(i=1,\dots,n)$, 
which are invariants coming from the linking numbers of $L$.

\begin{theorem}\label{thm11}
Let $L$ and $L'$ be oriented $n$-component 
virtual links with $n\geq 2$. 
Then the following are equivalent. 
\begin{enumerate}
\item
$L$ and $L'$ are $v\Delta$-equivalent. 

\item
$L$ and $L'$ are $v\Delta^\wedge$-equivalent. 

\item
$L$ and $L'$ are $v\sharp$-equivalent. 

\item
$p_i(L)=p_i(L')$ holds for any $i=1,\dots,n$. 
\end{enumerate}
\end{theorem}

\begin{theorem}\label{thm12}
Let $L$ and $L'$ be oriented $n$-component 
virtual links with $n\geq 2$. 
Then the following are equivalent. 
\begin{enumerate}
\item
$L$ and $L'$ are $v\Delta^\circ$-equivalent. 

\item
$L$ and $L'$ are $vp$-equivalent. 

\item
$\lambda_i(L)=\lambda_i(L')$ holds for any $i=1,\dots,n$. 
\end{enumerate}
\end{theorem}

For $X\in\{v\Delta, v\Delta^{\wedge}, v\Delta^{\circ}, v\sharp, vp\}$, 
we see that any two oriented virtual knots are $X$-equivalent. 
In particular, the $X$-move is an unknotting operation 
for oriented virtual knots. 
Therefore we can define the $X$-unknotting number 
${\rm u}_X(K)$ of an oriented virtual knot $K$, 
and will prove the following.

\begin{theorem}\label{thm-infinite}
For any $X\in\{v\Delta^\wedge, \ v\Delta^\circ, \ v\sharp, \ vp\}$ 
and positive integer $m$, 
there are infinitely many oriented virtual knots $K$ 
with ${\rm u}_X(K)=m$. 
\end{theorem}

This paper is organized as follows. 
In Section~\ref{sec2}, 
we divide virtualized $\Delta$-moves into eight types 
$v\Delta_1^{\wedge},\dots,v\Delta_4^{\wedge}$ and 
$v\Delta_1^{\circ},\dots,v\Delta_4^{\circ}$, 
and virtualized $\sharp$-moves into two types $v\sharp_1$ and $v\sharp_2$ 
according to the orientations of strings. 
We study their relations and prove Theorem~\ref{thm11}. 
Sections~\ref{sec3} and~\ref{sec4} are devoted to the proof of Theorem~\ref{thm12}. 
In Section~\ref{sec3}, 
we divide virtualized pass-moves into four types 
$vp_1,\dots,vp_4$ according to the string orientations. 
We study relations among $vp_i$'s and $v\Delta_j^{\circ}$'s, 
and prove the equivalence of (i) and (ii) in Theorem~\ref{thm12}. 
In Section~\ref{sec4}, 
we construct a family of oriented $n$-component virtual links, 
and prove that any oriented $n$-component virtual link $L$ 
is $v\Delta^{\circ}$-equivalent to a certain link belonging to the family. 
We define invariants $\lambda_i(L)$ $(i=1,\dots,n)$ 
by using the linking numbers of $L$, 
and prove the equivalence of (i) and (iii) in 
Theorem~\ref{thm12}. 
Finally, in Section~\ref{sec5}, 
we provide lower bounds for the $X$-distance 
between two oriented virtual knots 
for $X\in\{v\Delta^\wedge, \ v\Delta^\circ, \ v\sharp, \ vp\}$ 
in terms of their odd writhes and $n$-writhes. 
By using these lower bounds, 
we prove Theorem~\ref{thm-infinite}.

\section{Proof of Theorem~\ref{thm11}}\label{sec2}

A \textit{virtualized $\Delta$-move} or simply a \textit{$v\Delta$-move}  
is a local deformation on a link diagram 
as shown in Figure~\ref{ori-vdelta}. 
There are eight oriented types of virtualized $\Delta$-moves 
labeled by 
$v\Delta_1^\wedge,\dots,v\Delta_4^\wedge$ and 
$v\Delta_1^\circ,\dots,v\Delta_4^\circ$ 
as in the figure. 
The first four moves are collectively called \textit{$v\Delta^\wedge$-moves} 
and the latter \textit{$v\Delta^\circ$-moves}. 
We say that two oriented virtual links $L$ and $L'$ are 
\textit{$v\Delta$-, $v\Delta^\wedge$-}, and \textit{$v\Delta^\circ$-equivalent} 
if their diagrams are related by a finite sequence of 
$v\Delta$-, $v\Delta^\wedge$-, and $v\Delta^\circ$-moves 
(up to generalized Reidemeister moves), 
respectively. 

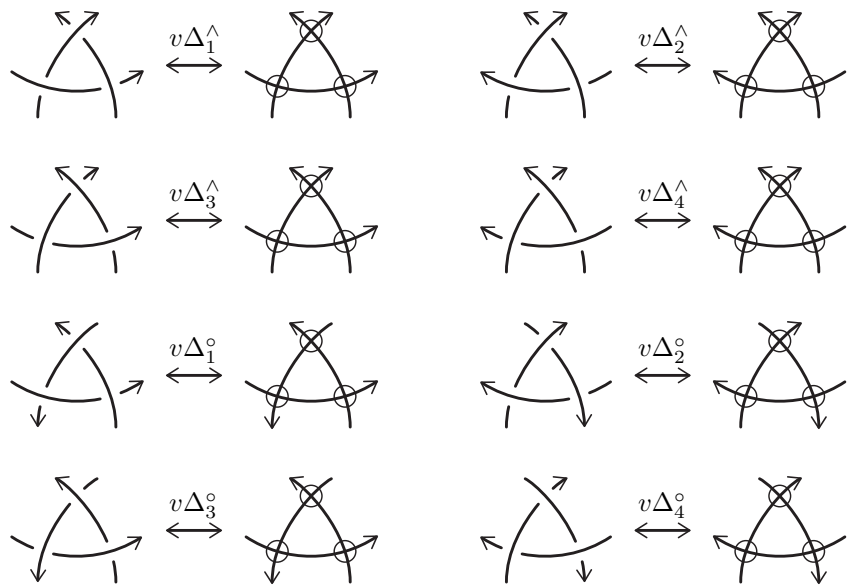
\begin{figure}[htbp]
  \centering
    \begin{overpic}[width=11cm]{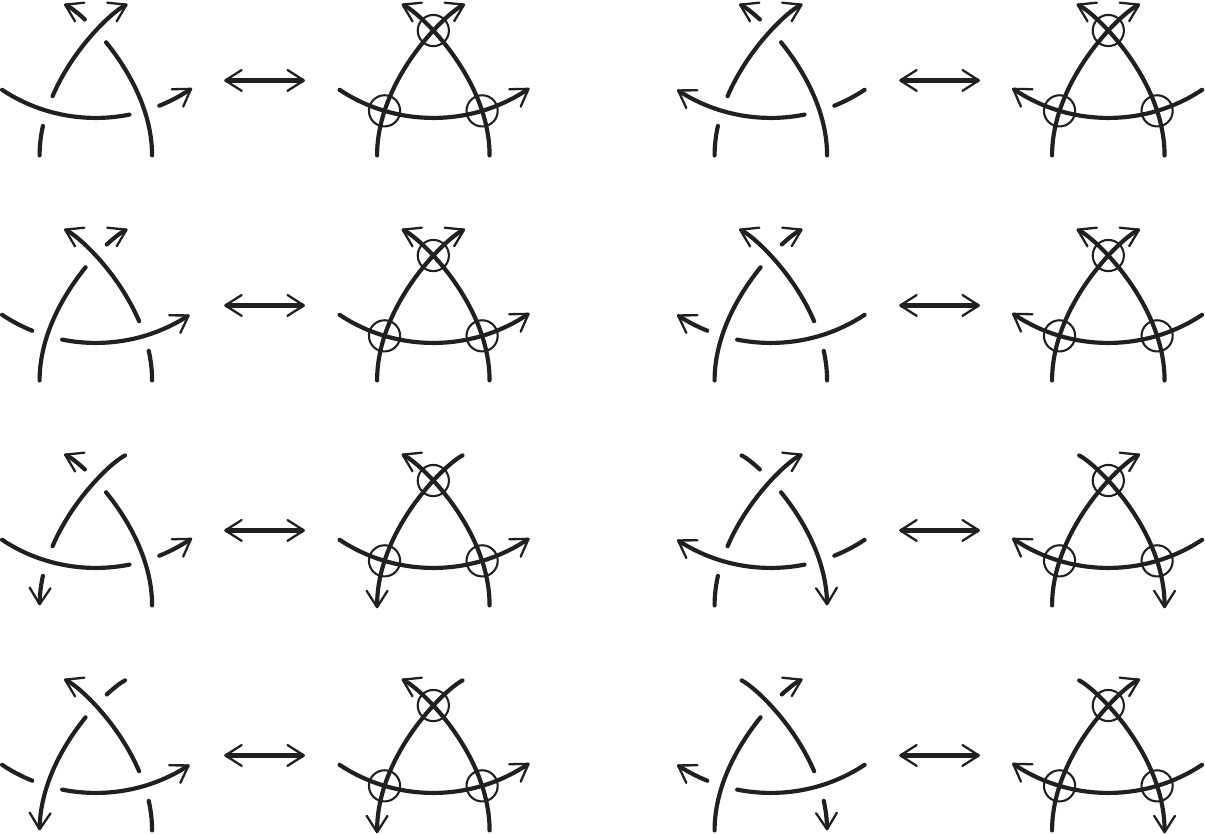}
      \put(59.5,203){$v\Delta_{1}^{\wedge}$}
      \put(235,203){$v\Delta_{2}^{\wedge}$}
      \put(59.5,144.5){$v\Delta_{3}^{\wedge}$}
      \put(235,144.5){$v\Delta_{4}^{\wedge}$}
      \put(59.5,86){$v\Delta_{1}^{\circ}$}
      \put(235,86){$v\Delta_{2}^{\circ}$}
      \put(59.5,27.5){$v\Delta_{3}^{\circ}$}
      \put(235,27.5){$v\Delta_{4}^{\circ}$}
    \end{overpic}
  \caption{Virtualized $\Delta$-moves}
  \label{ori-vdelta}
\end{figure}

\begin{lemma}\label{lem-cc}
For any $i\in\{1,\dots,4\}$, 
we have the following. 
\begin{enumerate}
\item
A crossing change is realized by 
a $v\Delta_i^\wedge$-move. 
\item
A crossing change is realized 
by a $v\Delta_i^\circ$-move. 
\end{enumerate}
\end{lemma}

\begin{proof}
(i) The sequence in the top row of Figure~\ref{pf-lem-cc-wedge} 
shows that a crossing change is realized by a combination of 
a $v\Delta_1^\wedge$-move and several generalized Reidemeister moves, 
where the symbol $\stackrel{\rm R}{\longleftrightarrow}$ means a combination of 
generalized Reidemeister moves. 
For a $v\Delta_2^\wedge$-move, 
we may use the above sequence with the orientations 
of all the strings reversed. 
See the second row of the figure. 
For $v\Delta_3^\wedge$- and $v\Delta_4^\wedge$-moves, 
we may use the sequences for 
$v\Delta_1^\wedge$- and $v\Delta_2^\wedge$-moves 
with opposite crossing information at every real crossing, respectively. 
See the third and bottom rows of the figure. 

\begin{figure}[t]
\centering
\vspace{1em}
    \begin{overpic}[width=11cm]{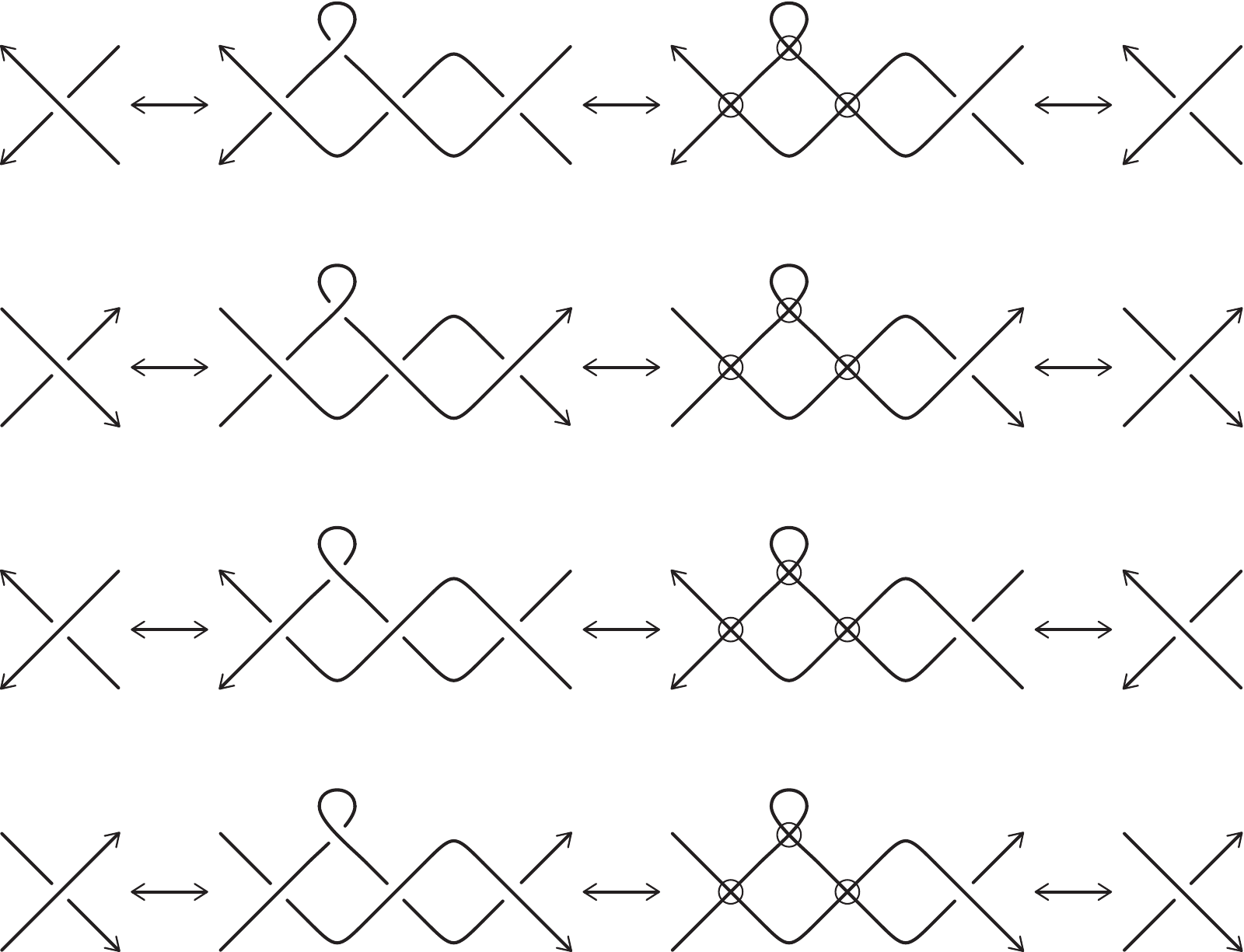}
      \put(0,245.5){\underline{$i=1$}}
      \put(0,179.5){\underline{$i=2$}}
      \put(0,113.5){\underline{$i=3$}}
      \put(0,47.5){\underline{$i=4$}}
      \put(39,217.6){R}
      \put(147,219.9){$v\Delta_{1}^{\wedge}$}
      \put(266.5,217.6){R}
      \put(39,151.6){R}
      \put(147,155.5){$v\Delta_{2}^{\wedge}$}
      \put(266.5,151.6){R}
      \put(39,85.6){R}
      \put(147,88.75){$v\Delta_{3}^{\wedge}$}
      \put(266.5,85.6){R}
      \put(39,19.6){R}
      \put(147,22){$v\Delta_{4}^{\wedge}$}
      \put(266.5,19.6){R}
    \end{overpic}
  \caption{Proof of Lemma~\ref{lem-cc}(i)}
  \label{pf-lem-cc-wedge}
\end{figure}

(ii) The sequence in Figure~\ref{pf-lem-cc-circ} 
shows that a crossing change is realized by a combination of 
a $v\Delta_1^\circ$-move and several generalized Reidemeister moves. 
We remark that it is obtained from the sequence for a $v\Delta_1^\wedge$-move 
given in (i) by reversing the orientation of the string pointed 
from the lower right to the upper left. 
We have a similar sequence for a $v\Delta_i^\circ$-move $(i=2,3,4)$ 
as shown in the figure. 
\end{proof}

\begin{figure}[htbp]
  \centering
  \vspace{1em}
    \begin{overpic}[width=11cm]{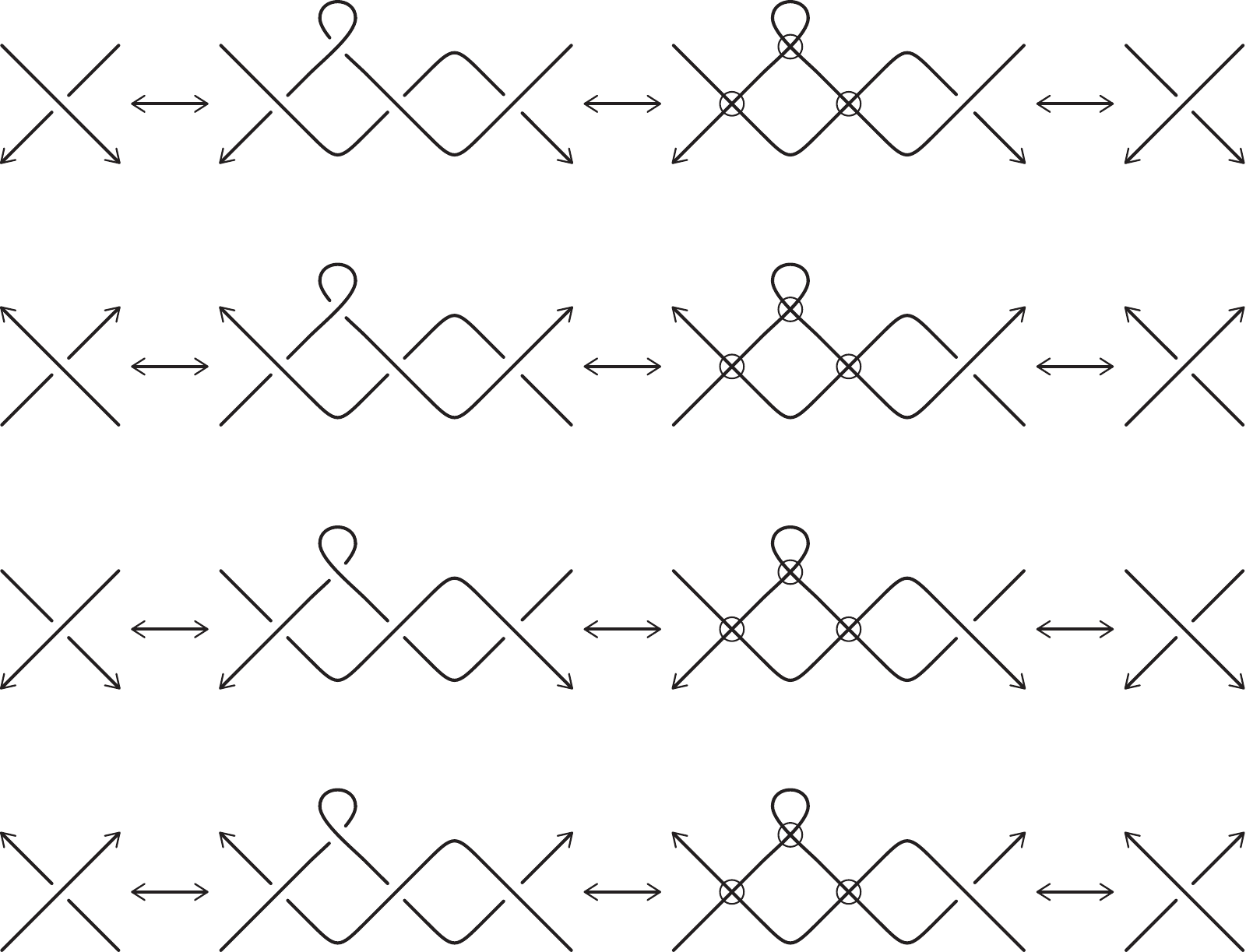}
      \put(0,245.5){\underline{$i=1$}}
      \put(0,179.5){\underline{$i=2$}}
      \put(0,113.5){\underline{$i=3$}}
      \put(0,47.5){\underline{$i=4$}}
      \put(39,217.6){R}
      \put(147,219.9){$v\Delta_{1}^{\circ}$}
      \put(266.5,217.6){R}
      \put(39,151.6){R}
      \put(147,155.5){$v\Delta_{2}^{\circ}$}
      \put(266.5,151.6){R}
      \put(39,85.6){R}
      \put(147,88.75){$v\Delta_{3}^{\circ}$}
      \put(266.5,85.6){R}
      \put(39,19.6){R}
      \put(147,22){$v\Delta_{4}^{\circ}$}
      \put(266.5,19.6){R}
    \end{overpic}
  \caption{Proof of Lemma~\ref{lem-cc}(ii)}
  \label{pf-lem-cc-circ}
\end{figure}
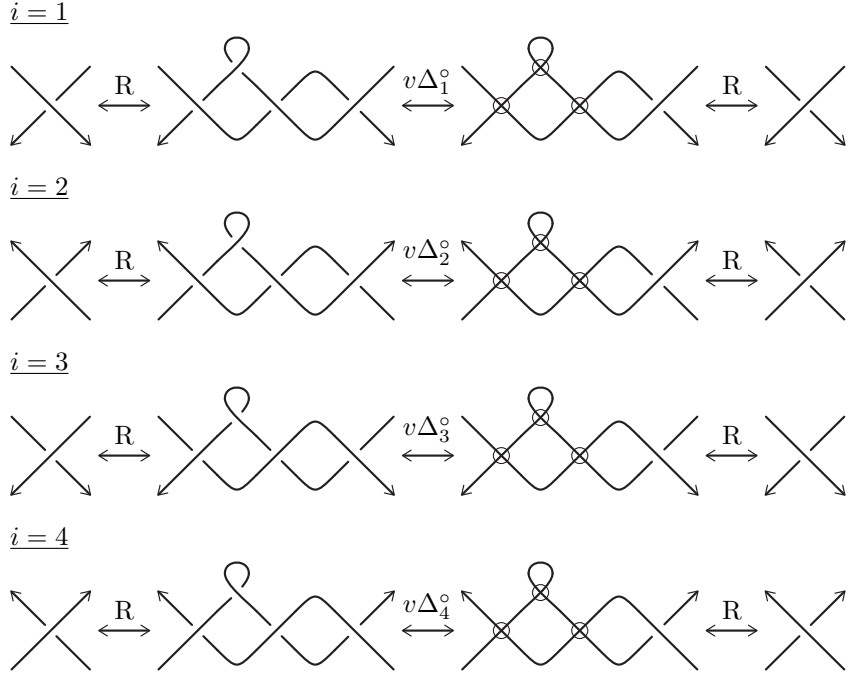

For two local moves $X$ and $Y$, 
we use the notation $X\Rightarrow Y$ 
if a $Y$-move is realized by a combination of 
$X$-moves and generalized Reidemeister moves. 

\begin{lemma}\label{lem-delta}
For the local moves 
$v\Delta_i^\wedge$ and $v\Delta_j^\circ$ 
$(i,j=1,\dots,4)$, 
we have the following. 
\begin{enumerate}
\item
$v\Delta_1^\wedge\Leftrightarrow
v\Delta_2^\wedge\Leftrightarrow
v\Delta_3^\wedge\Leftrightarrow
v\Delta_4^\wedge$. 
\item
$v\Delta_1^\circ\Leftrightarrow
v\Delta_2^\circ\Leftrightarrow
v\Delta_3^\circ\Leftrightarrow
v\Delta_4^\circ$. 
\item
$v\Delta_i^\wedge\Rightarrow
v\Delta_j^\circ$ for any $i$ and $j$. 
\end{enumerate}
\end{lemma}

\begin{proof}
(i) It is sufficient to prove 
\[
v\Delta_1^\wedge\Rightarrow v\Delta_2^\wedge 
\Rightarrow v\Delta_4^\wedge \Rightarrow 
v\Delta_3^\wedge\Rightarrow v\Delta_1^\wedge.\]
The sequence in the top row of Figure~\ref{pf-lem-delta1} shows that 
a $v\Delta_2^\wedge$-move is 
realized by a combination of a $\Delta$-move, 
a $v\Delta_1^\wedge$-move, and a generalized 
Reidemeister move. 
Since a $\Delta$-move is realized by 
a combination of two crossing changes and 
a generalized Reidemeister move, 
and a crossing change is realized by 
a $v\Delta_1^\wedge$-move by Lemma~\ref{lem-cc}(i), 
we have $v\Delta_1^\wedge\Rightarrow v\Delta_2^\wedge$. 
The remaining cases are proved similarly as shown in the figure, 
where $\stackrel{\rm cc}{\longleftrightarrow}$ 
means a combination of crossing changes at real crossings.

\begin{figure}[htbp]
\centering
\vspace{2em}
    \begin{overpic}[width=10cm]{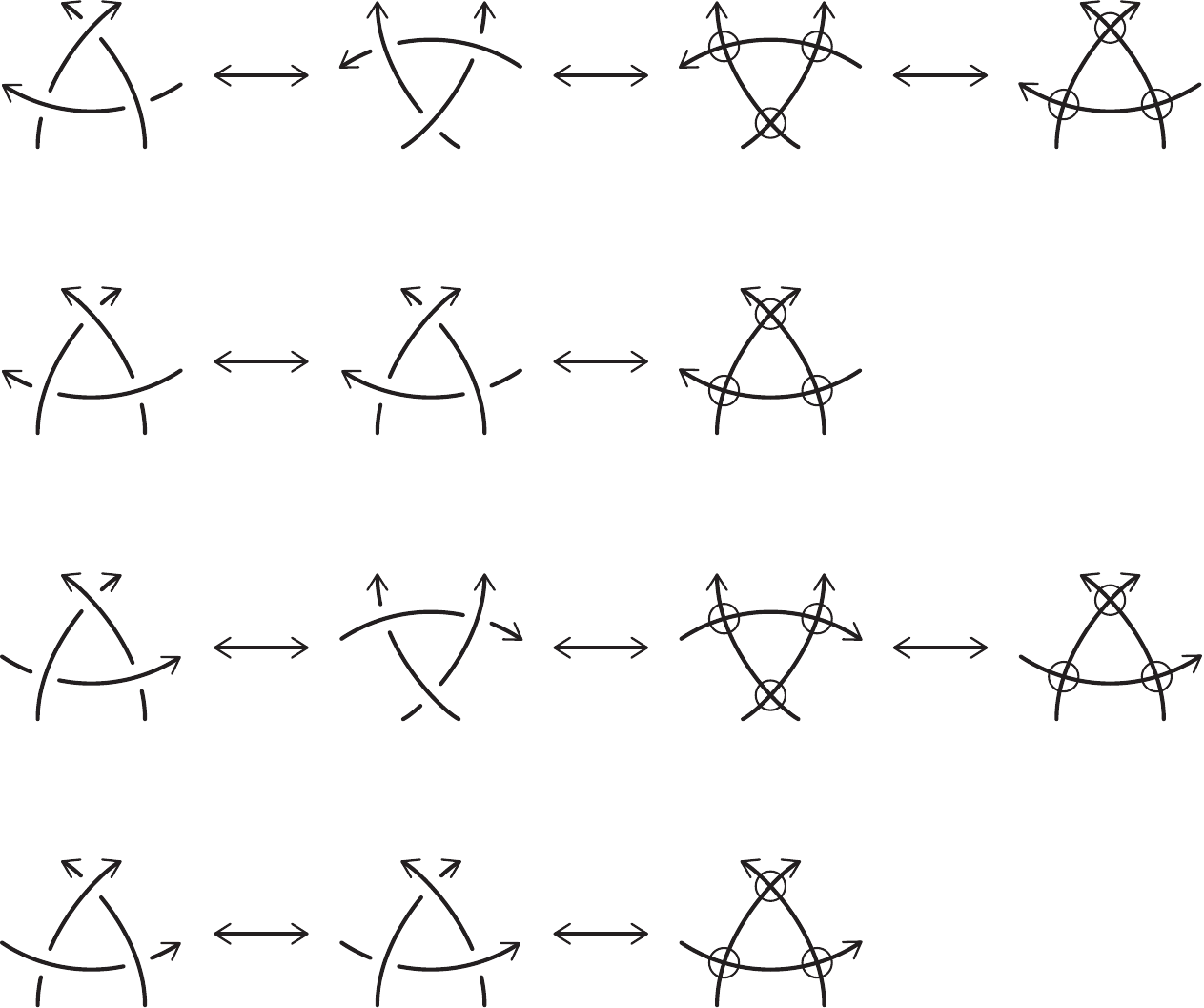}
      \put(0,249){\underline{$v\Delta_{1}^{\wedge}\Rightarrow v\Delta_{2}^{\wedge}$}}
      \put(0,181.4){\underline{$v\Delta_{2}^{\wedge}\Rightarrow v\Delta_{4}^{\wedge}$}}
      \put(0,113.8){\underline{$v\Delta_{4}^{\wedge}\Rightarrow v\Delta_{3}^{\wedge}$}}
      \put(0,46.2){\underline{$v\Delta_{3}^{\wedge}\Rightarrow v\Delta_{1}^{\wedge}$}}
      \put(57.6,224.8){$\Delta$}
      \put(133,227.3){$v\Delta_{1}^{\wedge}$}
      \put(219,224.8){R}
      \put(57.3,157.3){cc}
      \put(133,159.6){$v\Delta_{2}^{\wedge}$}
      \put(57.6,89.6){$\Delta$}
      \put(133,92){$v\Delta_{4}^{\wedge}$}
      \put(219,89.6){R}
      \put(57.6,22){cc}
      \put(133,24.5){$v\Delta_{3}^{\wedge}$}
    \end{overpic}
  \caption{Proof of Lemma~\ref{lem-delta}(i)}
  \label{pf-lem-delta1}
\end{figure}
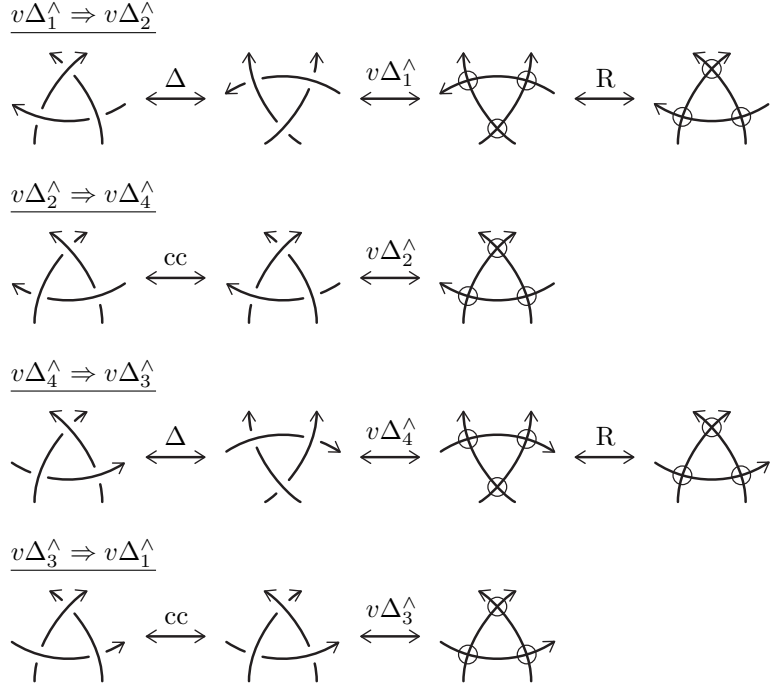

(ii) It is sufficient to prove 
\[
v\Delta_1^\circ\Rightarrow v\Delta_2^\circ 
\Rightarrow v\Delta_4^\circ \Rightarrow 
v\Delta_3^\circ\Rightarrow v\Delta_1^\circ.\]
Each of the implications can be proved by reversing the orientation of a certain string 
in a sequence given in (i). 
For example, 
Figure~\ref{pf-lem-delta2} shows $v\Delta_1^\circ\Rightarrow v\Delta_2^\circ$.

\begin{figure}[htbp]
\centering
    \begin{overpic}[width=10cm]{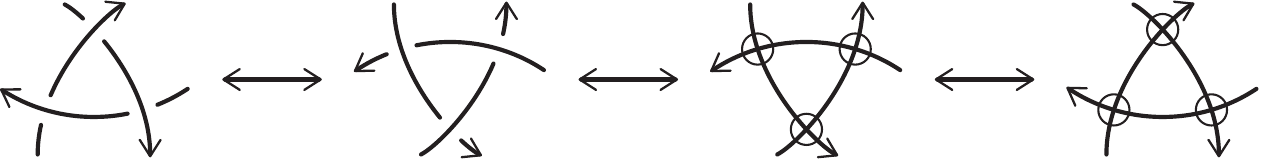}
      \put(57.6,22.3){$\Delta$}
      \put(133,24.8){$v\Delta_{1}^{\circ}$}
      \put(219,22.3){R}
    \end{overpic}
  \caption{Proof of $v\Delta_1^\circ\Rightarrow v\Delta_2^\circ$}
  \label{pf-lem-delta2}
\end{figure}

(iii) 
By (i) and (ii), it is sufficient to prove 
$v\Delta_2^\wedge\Rightarrow v\Delta_1^\circ$. 
Figure~\ref{pf-lem-delta3} shows that a $v\Delta_1^\circ$-move is realized by 
a combination of 
three crossing changes, three $v\Delta_2^\wedge$-moves, 
and several generalized Reidemesiter moves. 
Therefore we have $v\Delta_2^\wedge\Rightarrow v\Delta_1^\circ$ 
by Lemma~\ref{lem-cc}(ii). 
\end{proof}

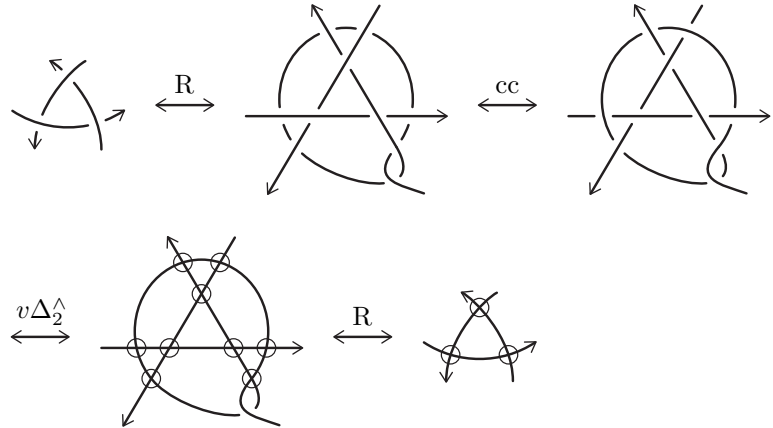
\begin{figure}[htbp]
\centering
    \begin{overpic}[width=10cm]{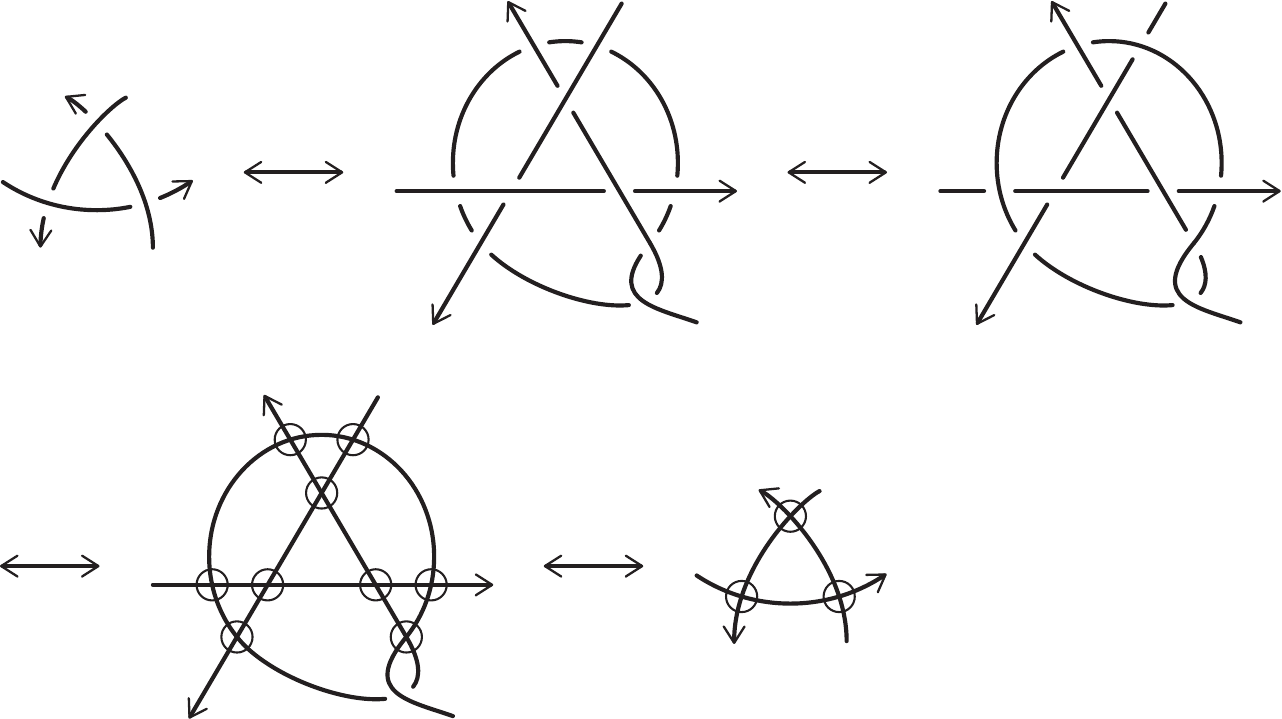}
      \put(61.3,125.8){R}
      \put(181.5,125.8){cc}
      \put(2,40.9){$v\Delta_{2}^{\wedge}$}
      \put(128,38.4){R}
    \end{overpic}
  \caption{Proof of $v\Delta_2^\wedge\Rightarrow v\Delta_1^\circ$}
  \label{pf-lem-delta3}
\end{figure}

A \textit{virtualized $\sharp$-move} or simply 
a \textit{$v\sharp$-move} is a local deformation on a link diagram 
as shown in Figure~\ref{vsharp}. 
There are two types of virtualized $\sharp$-moves 
labeled by $v\sharp_1$ and $v\sharp_2$ 
according to the sign of the real crossings 
as in the figure. 
We say that two oriented virtual links $L$ and $L'$ are 
\textit{$v\sharp$-equivalent} 
if their diagrams are related by a finite sequence of 
$v\sharp$-moves (up to generalized Reidemeister moves).

\begin{figure}[htbp]
\centering
    \begin{overpic}[width=7cm]{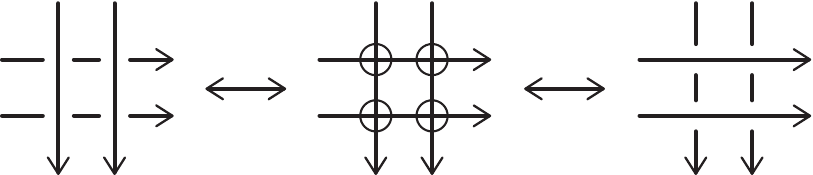}
      \put(54.5,28){$v\sharp_{1}$}
      \put(132.5,28){$v\sharp_{2}$}
    \end{overpic}
  \caption{Virtualized $\sharp$-moves}
  \label{vsharp}
\end{figure}

\begin{lemma}\label{lem-sharp}
$v\sharp_1\Leftrightarrow v\sharp_2$. 
More precisely, a $v\sharp_2$-move is realized by a $v\sharp_1$-move, 
and vice versa. 
\end{lemma}

\begin{proof}
Figure~\ref{pf-lem-sharp} shows that 
a $v\sharp_2$-move is realized by a combination of a $v\sharp_1$-move 
and several generalized Reidemeister moves. 
Thus we have $v\sharp_1\Rightarrow v\sharp_2$. 
The proof of $v\sharp_2\Rightarrow v\sharp_1$ is 
obtained from the above sequence by 
changing crossing information at every real crossing. 
\end{proof}

\begin{figure}[htbp]
\centering
    \begin{overpic}[width=12cm]{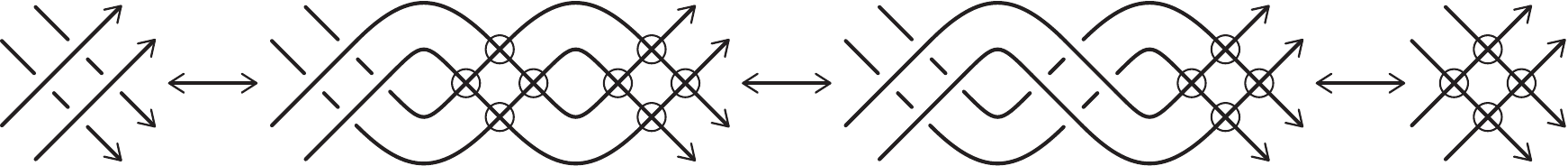}
      \put(43,22.3){R}
      \put(165,23.6){$v\sharp_{1}$}
      \put(293,22.3){R}
    \end{overpic}
  \caption{Proof of $v\sharp_1\Rightarrow v\sharp_2$}
  \label{pf-lem-sharp}
\end{figure}

\begin{lemma}\label{lem-cc2}
For any $i\in\{1,2\}$, 
a crossing change is realized by a $v\sharp_i$-move. 
\end{lemma}

\begin{proof}
The sequence in Figure~\ref{pf-lem-cc2} shows that 
a crossing change is realized by a combination of 
a $v\sharp_1$-move and several generalized Reidemeister moves. 
Therefore we have the conclusion by Lemma~\ref{lem-sharp}
\end{proof}

\begin{figure}[htbp]
  \centering
    \begin{overpic}[width=11cm]{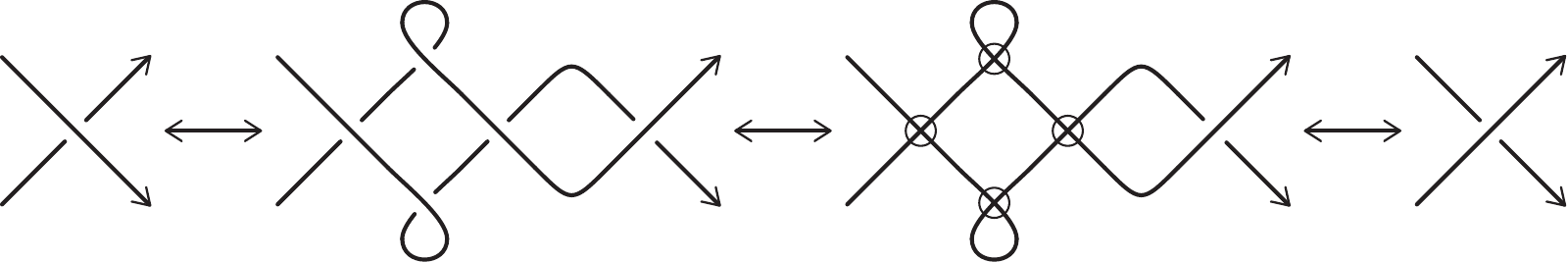}
      \put(39,30.3){R}
      \put(150,31.8){$v\sharp_{1}$}
      \put(266.5,30.3){R}
    \end{overpic}
  \caption{Proof of Lemma~\ref{lem-cc2} for $i=1$}
  \label{pf-lem-cc2}
\end{figure}

\begin{lemma}\label{lem-delta-sharp}
For any $i\in\{1,2\}$ and $j\in\{1,\dots,4\}$, 
we have 
$v\sharp_i\Leftrightarrow v\Delta_j^\wedge$. 
\end{lemma}

\begin{proof}
$(\Rightarrow)$ By Lemmas~\ref{lem-delta}(i) and \ref{lem-sharp}, 
it is sufficient to prove 
$v\sharp_1\Rightarrow v\Delta_1^\wedge$. 
The sequence in Figure~\ref{pf-lem-delta-sharp1} shows that 
a $v\Delta_1^\wedge$-move is realized by a combination of 
a crossing change, a $v\sharp_1$-move, and 
several generalized Reidemeister moves. 
Therefore we have $v\sharp_1\Rightarrow v\Delta_1^\wedge$
by Lemma~\ref{lem-cc2}.

\begin{figure}[htbp]
\centering
    \begin{overpic}[width=12cm]{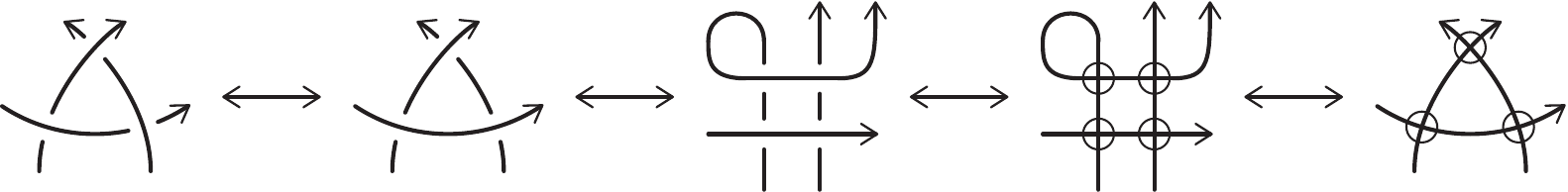}
      \put(54.5,25.2){cc}
      \put(132.5,25.2){R}
      \put(203,26.7){$v\sharp_{1}$}
      \put(278.5,25.2){R}
    \end{overpic}
  \caption{Proof of $v\sharp_1\Rightarrow v\Delta_1^{\wedge}$}
  \label{pf-lem-delta-sharp1}
\end{figure}
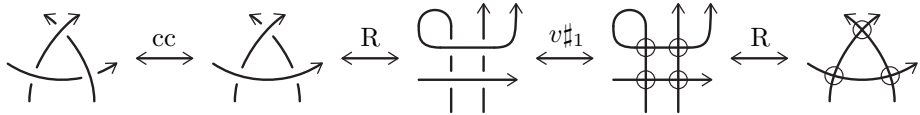

$(\Leftarrow)$
The sequence in Figure~\ref{pf-lem-delta-sharp2} shows that 
a $v\sharp_1$-move is realized by a combination of 
two crossing changes, a $v\Delta_1^\wedge$-move, 
a $v\Delta_4^\wedge$-move, and several generalized 
Reidemeister moves. 
Therefore we have $v\Delta_j^\wedge\Rightarrow v\sharp_i$ 
by Lemmas~\ref{lem-cc}(i), \ref{lem-delta}(i), 
and~\ref{lem-sharp}. 
\end{proof}

\begin{figure}[htbp]
\centering
    \begin{overpic}[width=12.5cm]{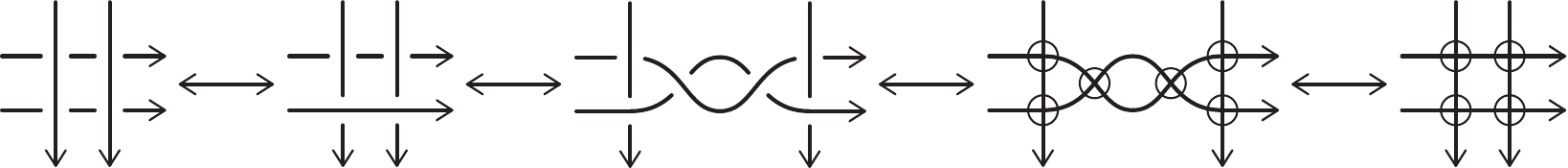}
      \put(46.8,23.5){cc}
      \put(113.5,23.5){R}
      \put(201,25.9){$v\Delta_{1}^{\wedge}$}
      \put(201,6.7){$v\Delta_{4}^{\wedge}$}
      \put(300.5,23.5){R}
    \end{overpic}
  \caption{A $v\sharp_{1}$-move is realized by $v\Delta_{1}^{\wedge}$- and $v\Delta_{4}^{\wedge}$-moves}
  \label{pf-lem-delta-sharp2}
\end{figure}

We are ready to prove Theorem~\ref{thm11}. 

\begin{proof}[Proof of {\rm Theorem~\ref{thm11}}]
\underline{(i)$\Leftrightarrow$(ii).} 
We have (i)$\Rightarrow$(ii) by Lemma~\ref{lem-delta}(iii), 
and (ii)$\Rightarrow$(i) by definition. 

\underline{(ii)$\Leftrightarrow$(iii).} 
This follows from Lemma~\ref{lem-delta-sharp} immediately. 

\underline{(i)$\Leftrightarrow$(iv).} 
This has been proved in~\cite[Theorem~1.5]{NNSW}. 
\end{proof}

\section{Proof of the equivalence of {\rm (i)} and {\rm (ii)} in Theorem~\ref{thm12}}\label{sec3}

A \textit{virtualized pass-move} or simply a \textit{vp-move} 
is a local move on a link diagram 
as shown in Figure~\ref{vpass}. 
There are four types of virtualized pass-moves 
labeled by $vp_1,\dots,vp_4$ as in the figure. 
We say that two oriented virtual links $L$ and $L'$ are 
\textit{$vp$-equivalent} 
if their diagrams are related by a finite sequence of 
$vp$-moves (up to generalized Reidemeister moves).

\begin{figure}[htbp]
\centering
    \begin{overpic}[width=10cm]{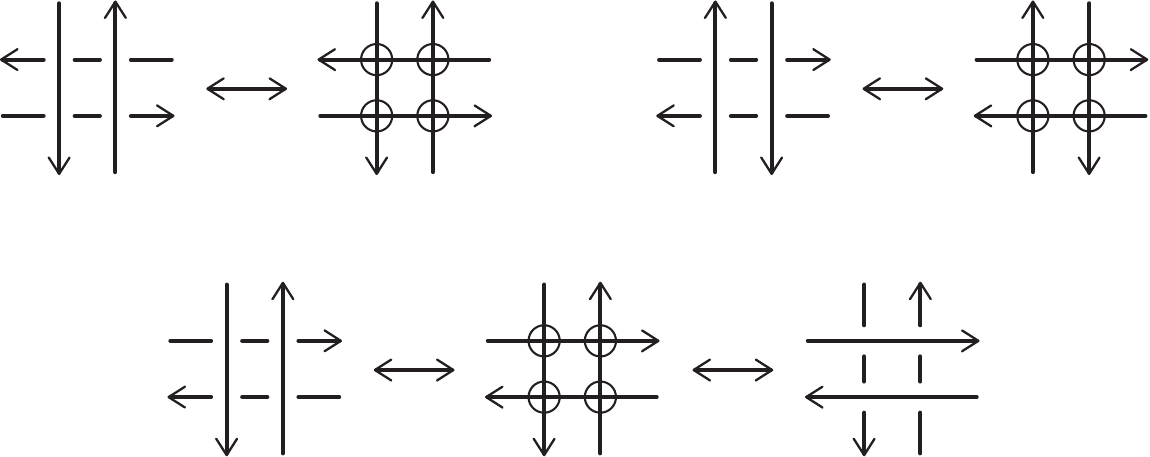}
      \put(54.5,97.5){$vp_{1}$}
      \put(216.5,97.5){$vp_{2}$}
      \put(95.5,28){$vp_{3}$}
      \put(174,28){$vp_{4}$}
    \end{overpic}
  \caption{Virtualized pass-moves}
  \label{vpass}
\end{figure}
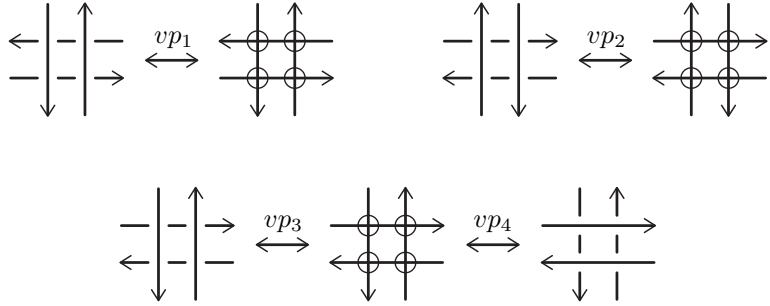

\begin{lemma}\label{lem-pass}
For any $i\ne j\in\{1,\dots,4\}$, 
we have $vp_i\Leftrightarrow vp_j$. 
More precisely, a $vp_i$-move is realized by a $vp_j$-move. 
\end{lemma}

\begin{proof}
Figure~\ref{pf-lem-pass} shows that 
a $vp_i$-move $(i=2,3)$ is realized by a combination of a $vp_1$-move 
and several generalized Reidemeister moves. 
The other cases are proved similarly. 
\end{proof}

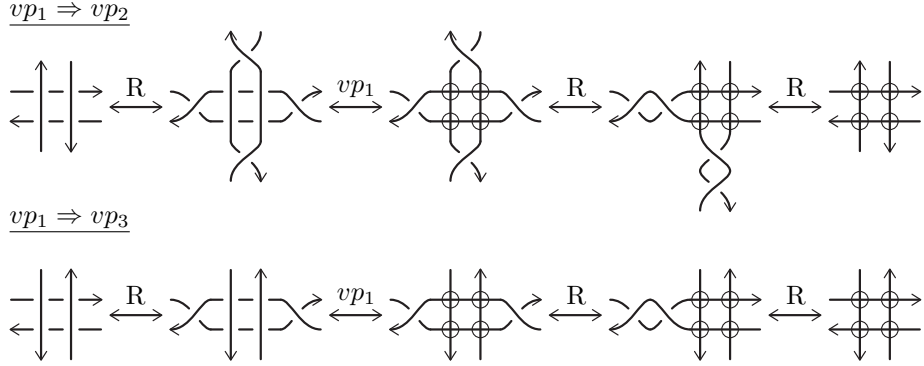
\begin{figure}[htbp]
\vspace{1em}
\centering
    \begin{overpic}[width=12cm]{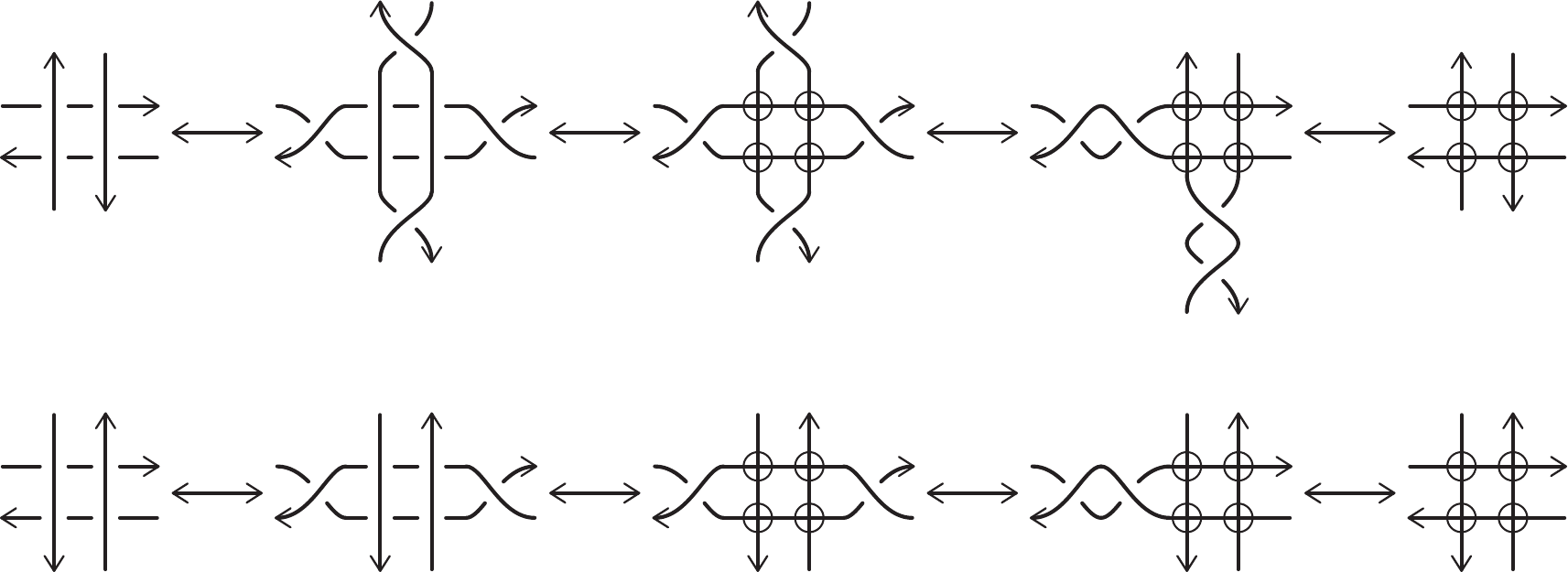}
      \put(0,130){\underline{$vp_{1}\Rightarrow vp_{2}$}}
      \put(0,52){\underline{$vp_{1}\Rightarrow vp_{3}$}}
      \put(43.5,99.9){R}
      \put(122.5,101.4){$vp_{1}$}
      \put(208.5,99.9){R}
       \put(290.7,99.9){R}
      \put(43.5,21.5){R}
      \put(122.5,23){$vp_{1}$}
      \put(208.5,21.5){R}
      \put(290.7,21.5){R}
    \end{overpic}
  \caption{A $vp_i$-move $(i=2,3)$ is realized by a $vp_1$-move}
  \label{pf-lem-pass}
\end{figure}

\begin{lemma}\label{lem-cc3}
For any $i\in\{1,\dots,4\}$, 
a crossing change is realized by a $vp_i$-move. 
\end{lemma}

\begin{proof}
The sequence in Figure~\ref{pf-lem-cc3} shows that a crossing change 
is realized by a combination of a $vp_1$-move and several generalized 
Reidemeister moves. 
Therefore we have the conclusion by Lemma~\ref{lem-pass}. 
\end{proof}

\begin{figure}[htbp]
  \centering
    \begin{overpic}[width=11cm]{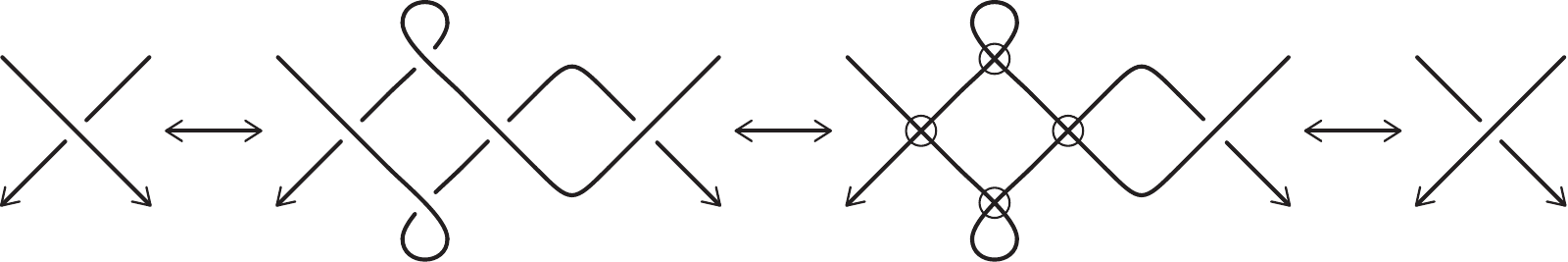}
      \put(39,30.3){R}
      \put(150,31.8){$vp_{1}$}
      \put(266.5,30.3){R}    
    \end{overpic}
  \caption{Proof of Lemma~\ref{lem-cc3} for $i=1$}
  \label{pf-lem-cc3}
\end{figure}

\begin{lemma}\label{lem-delta-pass}
For any $i, j\in\{1,\dots,4\}$, 
we have $vp_i\Leftrightarrow v\Delta_j^\circ$. 
\end{lemma}

\begin{proof}
$(\Rightarrow)$ 
By Lemmas~\ref{lem-delta}(ii) and \ref{lem-pass}, 
it is sufficient to prove 
$vp_1\Rightarrow v\Delta_1^\circ$. 
The sequence in Figure~\ref{pf-lem-delta-pass1} shows that 
a $v\Delta_1^\circ$-move is realized by a combination of 
a crossing change, a $vp_1$-move, and 
several generalized Reidemeister moves. 
Therefore we have $vp_1\Rightarrow v\Delta_1^\circ$
by Lemma~\ref{lem-cc3}. 

\begin{figure}[htbp]
\centering
    \begin{overpic}[width=12cm]{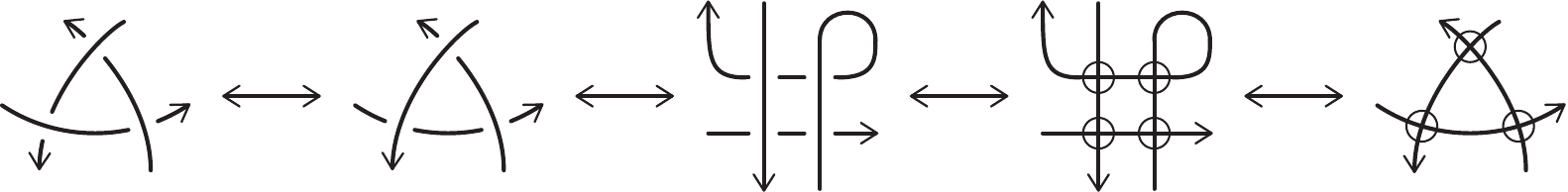}
      \put(54.5,25.2){cc}
      \put(132.5,25.2){R}
      \put(203,26.7){$vp_{1}$}
      \put(278.5,25.2){R}
    \end{overpic}
  \caption{Proof of $vp_1\Rightarrow v\Delta_1^{\circ}$}
  \label{pf-lem-delta-pass1}
\end{figure}
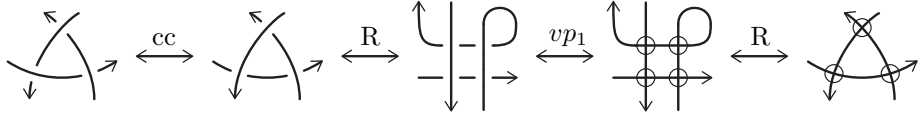

$(\Leftarrow)$ 
The sequence in Figure~\ref{pf-lem-delta-pass2} shows that 
a $vp_1$-move is realized by a combination of 
two crossing changes, a $v\Delta_1^\circ$-move, 
a $v\Delta_3^\circ$-move, and several generalized 
Reidemeister moves. 
Therefore we have $v\Delta_j^\circ\Rightarrow vp_i$
by Lemmas~\ref{lem-cc}(ii), \ref{lem-delta}(ii), 
and~\ref{lem-pass}. 
\end{proof}

\begin{figure}[t]
\centering
    \begin{overpic}[width=12.5cm]{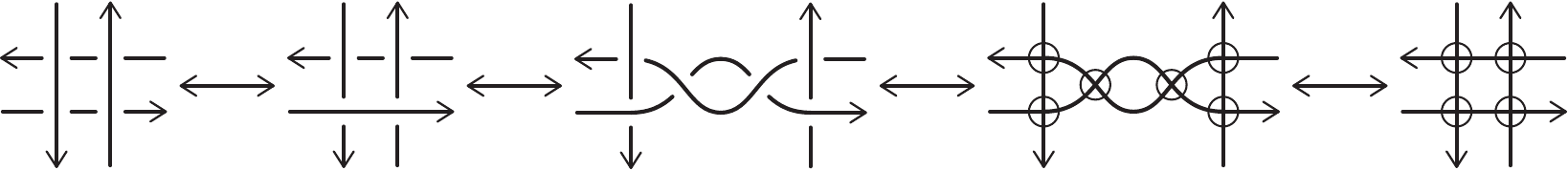}
      \put(46.8,23.5){cc}
      \put(113.5,23.5){R}
      \put(201,25.9){$v\Delta_{1}^{\circ}$}
      \put(201,6.7){$v\Delta_{3}^{\circ}$}
      \put(300.5,23.5){R}
    \end{overpic}
  \caption{A $vp_{1}$-move is realized by $v\Delta_{1}^{\circ}$- and $v\Delta_{3}^{\circ}$-moves}
  \label{pf-lem-delta-pass2}
\end{figure}
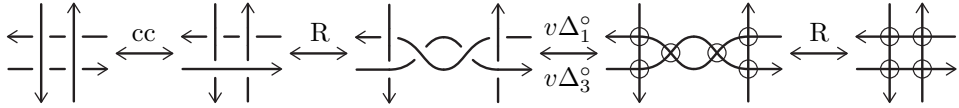

\begin{proof}[Proof of {\rm (i)$\Leftrightarrow$(ii) in Theorem~\ref{thm12}}]
This follows from Lemma~\ref{lem-delta-pass} immediately. 
\end{proof}

\section{Proof of the equivalence of {\rm (i)} and {\rm (iii)} in Theorem~\ref{thm12}}\label{sec4}

A \textit{Gauss diagram} of an oriented $n$-component link diagram is 
a union of $n$ oriented circles regarded as the preimage of 
the immersed circles with chords connecting two points in 
the preimage of each real crossing. 
Each chord is equipped with the sign of the corresponding real crossing, 
and it is oriented from the overcrossing to the undercrossing.  

A $v\Delta_i^\circ$-move $(i=1,\dots,4)$ on a link diagram 
is described by deleting/adding three chords on a Gauss diagram 
as shown in Figure~\ref{vDelta-circ-Gauss}, 
where the signs of the chords are the same.

\begin{figure}[htbp]
\centering
    \begin{overpic}[width=11cm]{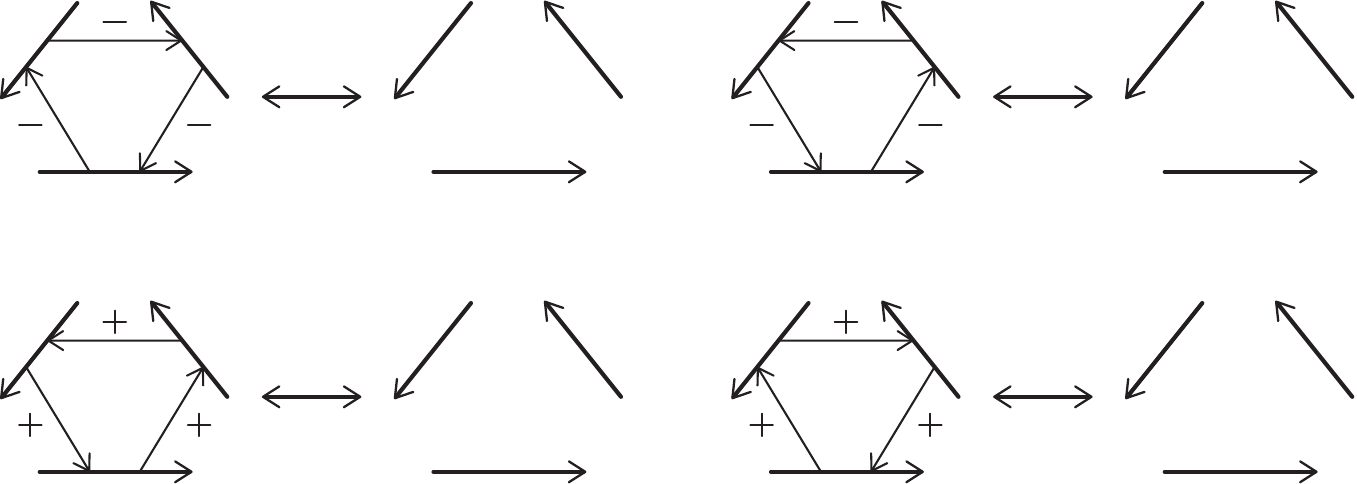}
      \put(63,96.3){$v\Delta_{1}^{\circ}$}
      \put(231,96.3){$v\Delta_{2}^{\circ}$}
      \put(63,27.2){$v\Delta_{3}^{\circ}$}
      \put(231,27.2){$v\Delta_{4}^{\circ}$}
    \end{overpic}
  \caption{A $v\Delta_{i}^{\circ}$-move $(i=1,\dots,4)$ on a Gauss diagram}
  \label{vDelta-circ-Gauss}
\end{figure}
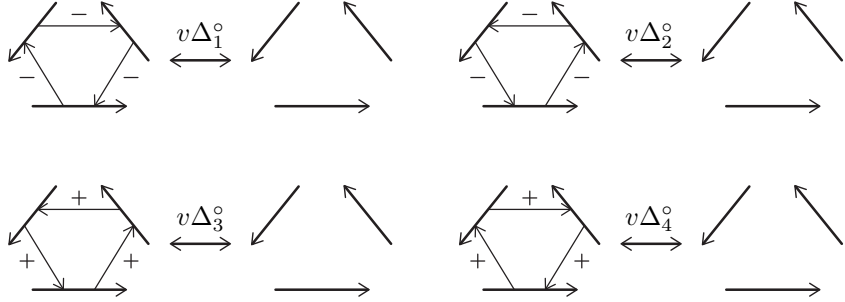

A \textit{forbidden detour move}~\cite{CMG,YI} 
or a \textit{fused move}~\cite{ABMW} on a link diagram 
is described by exchanging the positions of 
two consecutive initial and terminal endpoints of chords 
on a Gauss diagram. 
There are four types according to the signs of the chords, 
where we label them by $FD_1,\dots,FD_4$ 
as shown in Figure~\ref{fd-Gauss}.

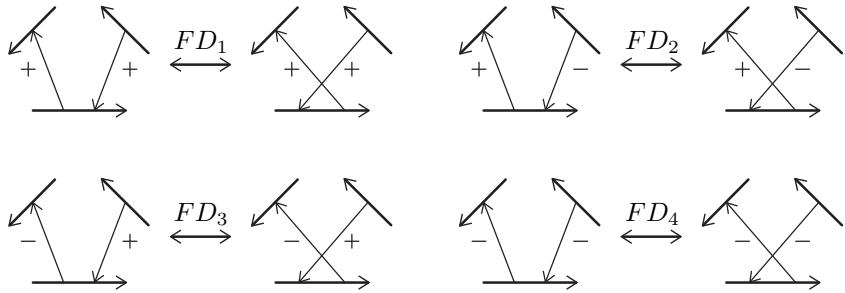
\begin{figure}[t]
\centering
    \begin{overpic}[width=11cm]{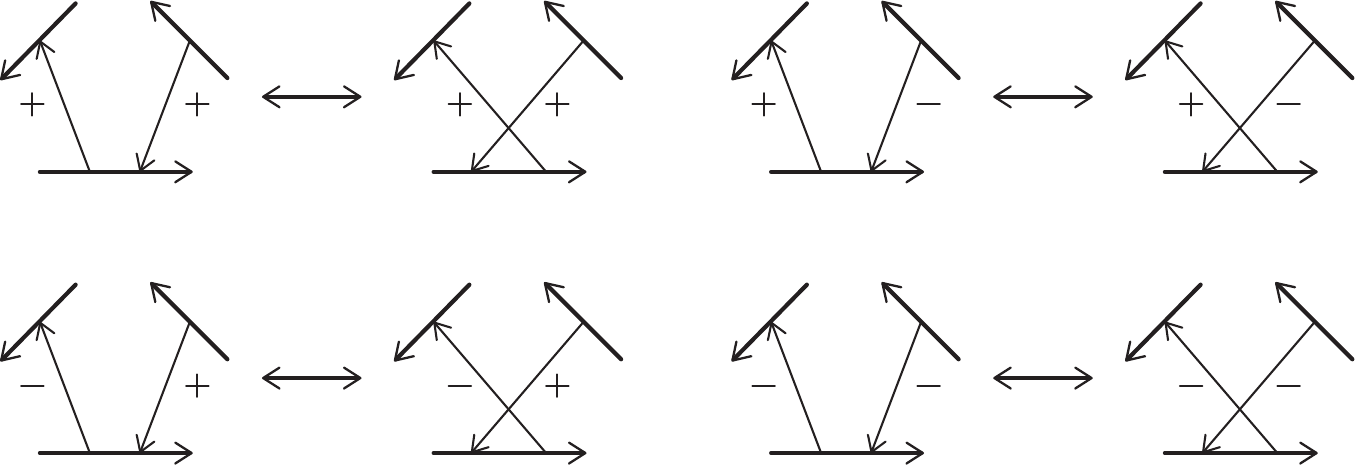}
      \put(62,90.75){$FD_{1}$}
      \put(231,90.75){$FD_{2}$}
      \put(62,26){$FD_{3}$}
      \put(231,26){$FD_{4}$}
    \end{overpic}
  \caption{Forbidden detour moves on Gauss diagrams}
  \label{fd-Gauss}
\end{figure}

\begin{lemma}\label{lem-fd}
For any $i\ne j\in\{1,\dots,4\}$, 
we have 
$FD_i\Leftrightarrow FD_j$. 
\end{lemma}

\begin{proof}
It is sufficient to prove 
\[FD_1\Rightarrow FD_2\Rightarrow FD_4\Rightarrow FD_3
\Rightarrow FD_1.\]
The sequence in the top of Figure~\ref{pf-lem-fd} 
shows $FD_1\Rightarrow FD_2$ for $\varepsilon=+1$ 
and $FD_4\Rightarrow FD_3$ for $\varepsilon=-1$. 
We remark that two Reidemeister moves II 
appear in this sequence. 
Similarly, the sequence in the bottom of the figure 
shows $FD_3\Rightarrow FD_1$ for $\varepsilon=+1$ 
and $FD_2\Rightarrow FD_4$ for $\varepsilon=-1$. 
\end{proof}

\begin{figure}[htbp]
\centering
\vspace{2em}
    \begin{overpic}[width=12cm]{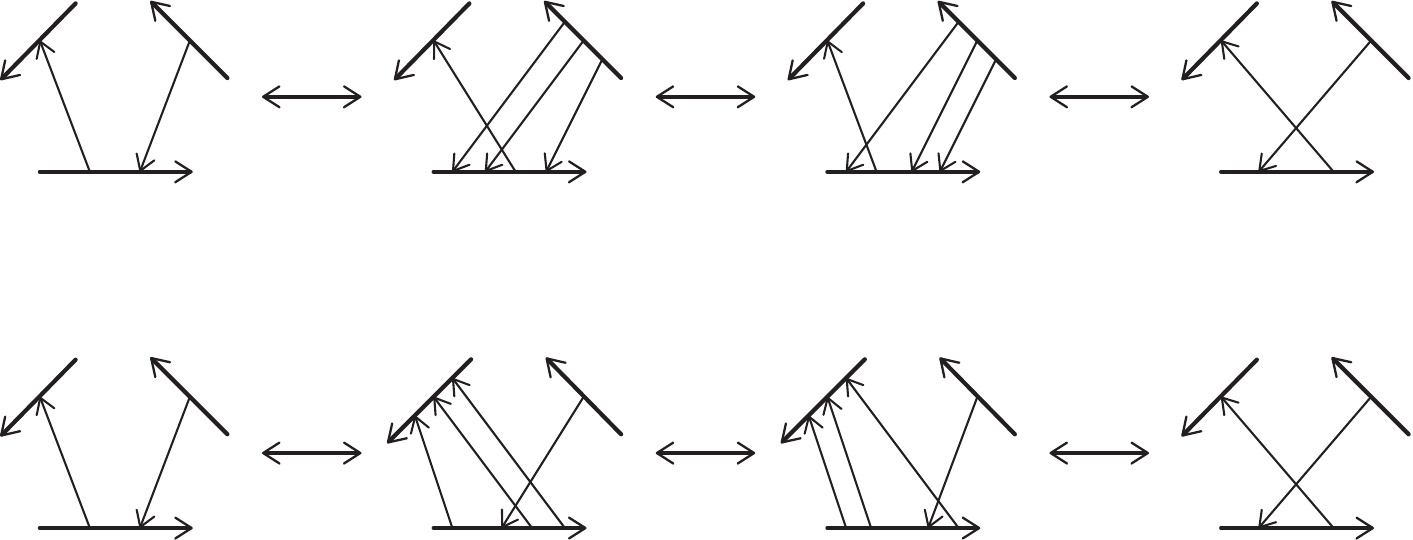}
      \put(0,144.5){\underline{$FD_1\Rightarrow FD_2$ for $\varepsilon=+1$ 
and $FD_4\Rightarrow FD_3$ for $\varepsilon=-1$}}
      \put(160.5,122){$FD_{1}$}
      \put(155,113){\scriptsize{$(\e=+1)$}}
      \put(160.5,95.5){$FD_{4}$}
      \put(155,86.5){\scriptsize{$(\e=-1)$}}
      \put(72.3,112){R}
      \put(262.3,112){R}
      \put(7,104){$\e$}
      \put(43.5,104){$-\e$}
      \put(104,105){$\e$}
      \put(116,117){$-\e$}
      \put(129,98){$\e$}
      \put(139.5,98){$-\e$}
      \put(197,105){$\e$}
      \put(211,117){$-\e$}
      \put(218.5,98){$\e$}
      \put(235.5,98){$-\e$}
      \put(296,104){$\e$}
      \put(323,104){$-\e$}
      \put(0,55){\underline{$FD_3\Rightarrow FD_1$ for $\varepsilon=+1$ 
and $FD_2\Rightarrow FD_4$ for $\varepsilon=-1$}}
      \put(160.5,36){$FD_{3}$}
      \put(155,27){\scriptsize{$(\e=+1)$}}
      \put(160.5,9.4){$FD_{2}$}
      \put(155,0.4){\scriptsize{$(\e=-1)$}}
      \put(72.3,25.9){R}
      \put(262.3,25.9){R}
      \put(7,17){$\e$}
      \put(45,17){$\e$}
      \put(97,9){$\e$}
      \put(107,9){$-\e$}
      \put(121,27){$\e$}
      \put(135,17){$\e$}
      \put(193.5,9){$\e$}
      \put(208.5,9){$-\e$}
      \put(217,27){$\e$}
      \put(234,17){$\e$}
      \put(296,17){$\e$}
      \put(324,17){$\e$}
    \end{overpic}
  \caption{Proof of Lemma~\ref{lem-fd}}
  \label{pf-lem-fd}
\end{figure}

\begin{lemma}\label{lem-delta-fd}
For any $i,j\in\{1,\dots,4\}$, 
we have 
$v\Delta_i^\circ \Rightarrow FD_j$. 
\end{lemma}

\begin{proof}
The sequence in Figure~\ref{pf-lem-delta-fd} shows that 
an $FD_1$-move is realized by a combination of 
a $v\Delta_4^\circ$-move, a $v\Delta_3^\circ$-move, 
and two Reidemeister moves II. 
By Lemmas~\ref{lem-delta}(ii) and \ref{lem-fd}, 
we have the conclusion. 
\end{proof}

\begin{figure}[htbp]
\centering
    \begin{overpic}[width=10cm]{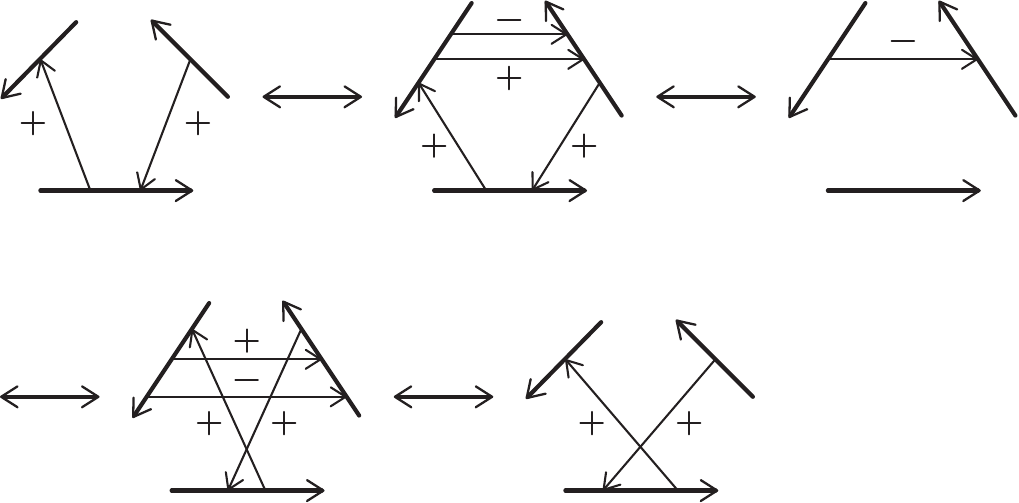}
      \put(83,118.4){R}
      \put(188.5,120.9){$v\Delta_{4}^{\circ}$}
      \put(5,37.2){$v\Delta_{3}^{\circ}$}
      \put(120.5,34.6){R}
    \end{overpic}
  \caption{An $FD_1$-move is realized by a $v\Delta_4^\circ$-move and a 
$v\Delta_3^\circ$-move}
  \label{pf-lem-delta-fd}
\end{figure}
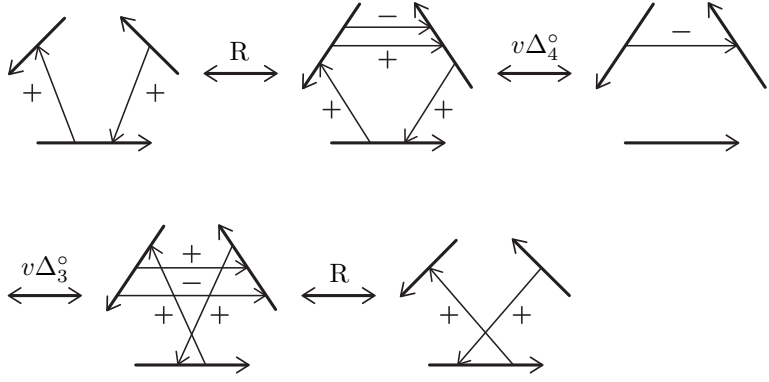

A \textit{forbidden move}~\cite{GPV} on a link diagram 
is described by exchanging the positions of 
two consecutive endpoints of chords on a Gauss diagram 
which are both initial or both terminal. 
There are six types according to the signs and orientations of the chords, 
where we label them by $F_1,\dots,F_6$ 
as shown in Figure~\ref{forbidden}. 
We say that two oriented virtual links $L$ and $L'$ are 
\textit{$F$-equivalent} 
if their diagrams are related by a finite sequence of 
forbidden moves (up to generalized Reidemeister moves). 
We remark that any two oriented virtual knots are $F$-equivalent 
\cite{Kan, Nel}.

\begin{figure}[htbp]
\centering
    \begin{overpic}[width=11cm]{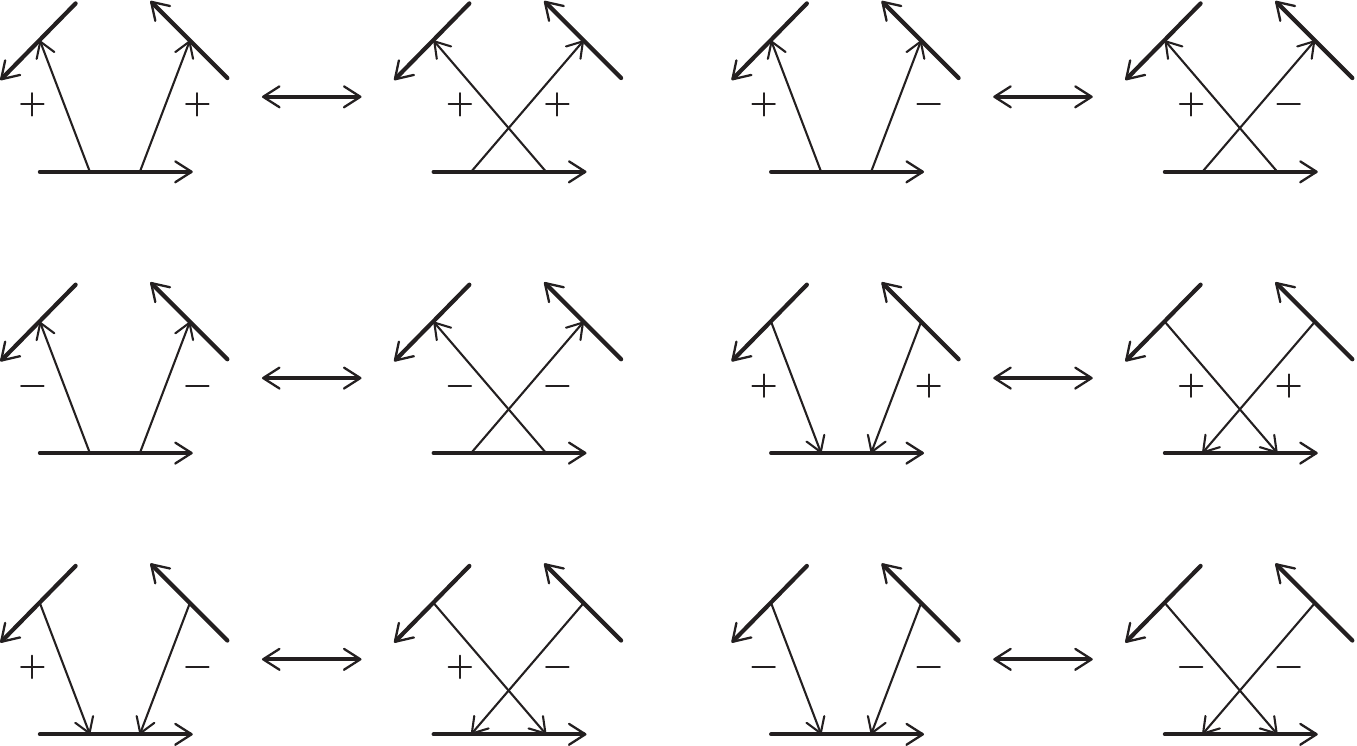}
      \put(67.5,156){$F_{1}$}
      \put(235.5,156){$F_{2}$}
      \put(67.5,91){$F_{3}$}
      \put(235.5,91){$F_{4}$}
      \put(67.5,26){$F_{5}$}
            \put(235.5,26){$F_{6}$}
    \end{overpic}
  \caption{Forbidden moves}
  \label{forbidden}
\end{figure}
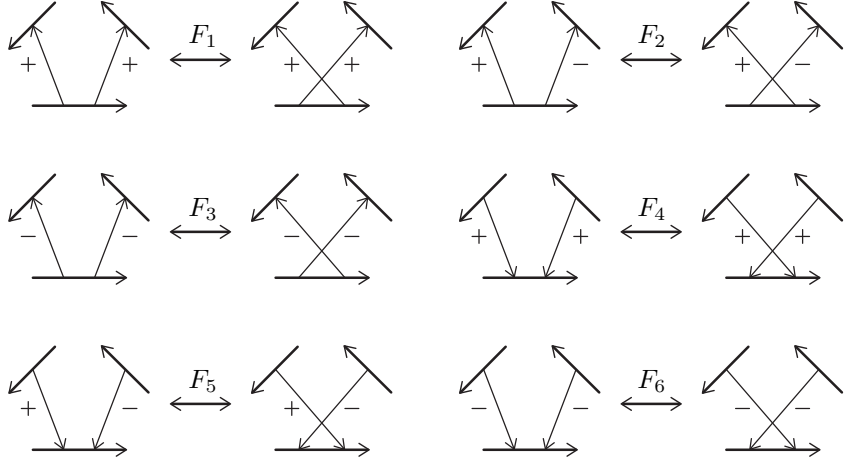

\begin{lemma}\label{lem-delta-f}
For any $i\in\{1,\dots,4\}$ and $j\in\{1,\dots,6\}$, 
we have 
$v\Delta_i^\circ \Rightarrow F_j$. 
\end{lemma}

\begin{proof}
We first consider the case $j=1$. 
The sequence in Figure~\ref{pf-lem-delta-f} shows that 
an $F_1$-move is realized by a combination of 
two crossing changes and an $FD_2$-move. 
Note that a crossing change at a real crossing on a link diagram 
is described by changing 
the sign and orientation of the corresponding chord on a Gauss diagram.
Therefore we have $v\Delta_i^\circ\Rightarrow F_1$ 
by Lemmas~\ref{lem-cc}(ii) and \ref{lem-delta-fd} 
for any $i$. 

\begin{figure}[htbp]
\centering
    \begin{overpic}[width=12cm]{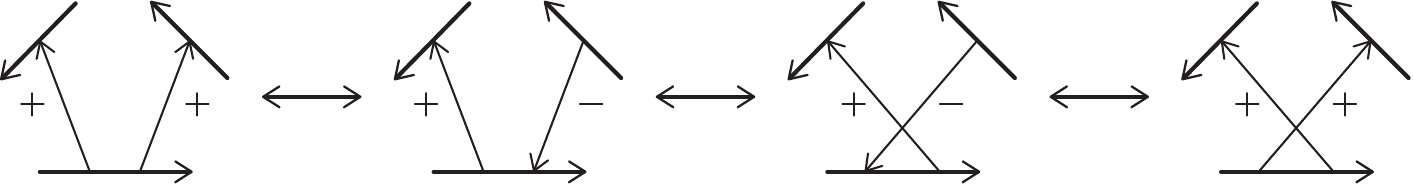}
      \put(71,25.5){cc}
      \put(160.5,27){$FD_{2}$}
      \put(261.5,25.5){cc}
    \end{overpic}
  \caption{Proof of $v\Delta_i^\circ\Rightarrow F_1$}
  \label{pf-lem-delta-f}
\end{figure}

Similarly, 
an $F_j$-move for $j\in\{2,\dots,6\}$ is realized by a combination of 
two crossing changes and an $FD_k$-move for some $k\in \{1,\ldots,4\}$. 
Thus we have the conclusion 
by Lemmas~\ref{lem-cc}(ii) and \ref{lem-delta-fd}. 
\end{proof}

In the remaining of this section, 
let $n$ be an integer with $n\geq 2$. 
For $n-1$ integers $a_2,\dots,a_n$, 
let $H(a_2,\dots,a_n)=\bigcup_{i=1}^n H_i$ be the Gauss diagram 
of an oriented $n$-component virtual link such that 
\begin{enumerate}
\item
$H(a_2,\dots,a_n)$ has no self-chords, 
\item
there are no nonself-chords between 
$H_i$ and $H_j$ ($2\leq i<j\leq n$), 
\item
if $a_i=0$, then there are no nonself-chords 
between $H_1$ and $H_i$, 
\item
if $a_i>0$, then there are $a_i$ parallel nonself-chords 
oriented from $H_1$ to $H_i$ with positive signs,  
\item
if $a_i<0$, then there are $-a_i$ parallel nonself-chords 
oriented from $H_1$ to $H_i$ with negative signs, and 
\item
along $H_1$ with respect to the orientation,  
we meet the endpoints of the chords 
between $H_1$ and $H_i$ before those between $H_1$ and $H_j$
$(2\leq i<j\leq n)$. 
\end{enumerate}

Figure~\ref{ex-H} shows the Gauss diagram $H(2,0,4,-3)$ with $n=5$. 
Let $M(a_2,\dots,a_n)$ be the $n$-component virtual link 
presented by $H(a_2,\dots,a_n)$. 

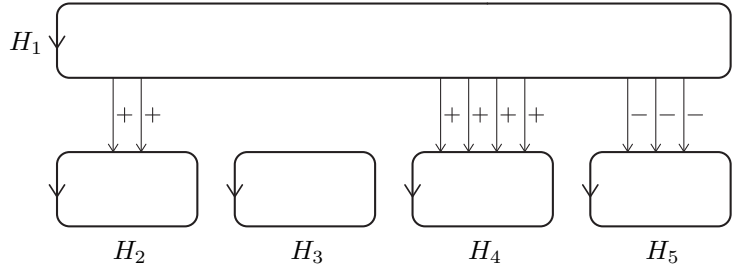
\begin{figure}[htbp]
\centering
    \begin{overpic}[width=9cm]{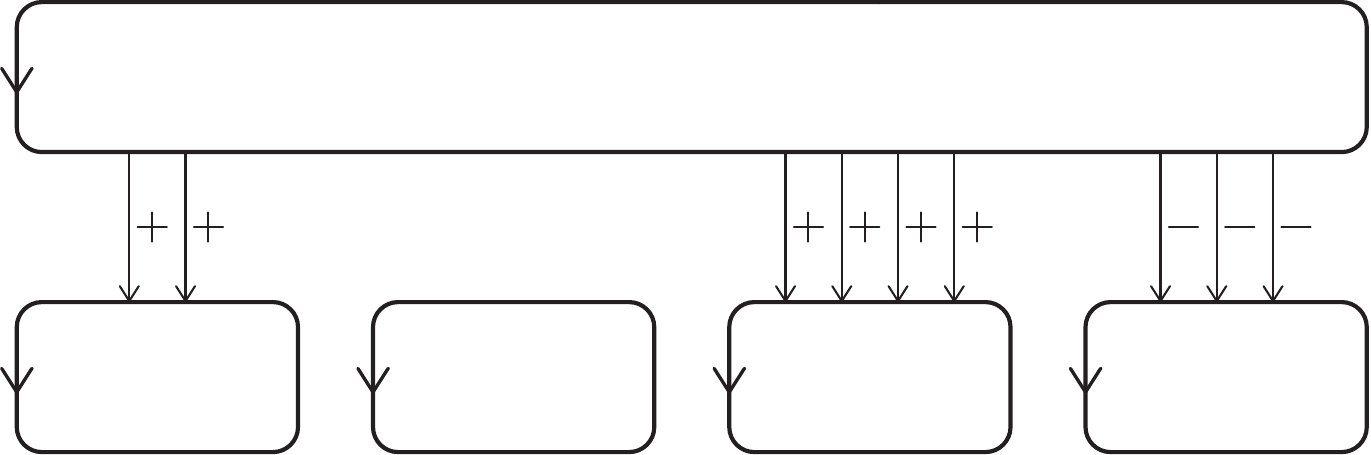}
      \put(-14.8,66.9){$H_{1}$}
      \put(24,-12){$H_{2}$}
      \put(90.5,-12){$H_{3}$}
      \put(157,-12){$H_{4}$}
      \put(223.5,-12){$H_{5}$}
    \end{overpic}
  \vspace{1em}
  \caption{The Gauss diagram $H(2,0,4,-3)$}
  \label{ex-H}
\end{figure}

\begin{proposition}\label{prop-standard}
Any oriented $n$-component virtual link $L$ is 
$v\Delta^\circ$-equivalent to 
$M(a_2,\dots,a_n)$ for some 
$a_2,\dots,a_n\in{\Z}$. 
\end{proposition}

To prove this proposition, 
we prepare the following lemma.

\begin{lemma}\label{lem-replacement}
Let $G=\bigcup_{i=1}^n G_i$ be a Gauss diagram of 
an oriented $n$-component virtual link. 
Then any nonself-chord oriented from $G_i$ to $G_j$ 
$(2\leq i\ne j\leq n)$ 
with sign $\varepsilon$ 
can be replaced with a pair of nonself-chords 
one of which is oriented from $G_1$ to $G_i$ with sign $-\varepsilon$ 
and the other is from $G_1$ to $G_j$ with sign $\varepsilon$ 
by a combination of $v\Delta_k^\circ$-moves 
and a Reidemeister move {\rm II} for any $k\in\{1,\dots,4\}$.
\end{lemma}

\begin{proof}
The sequence in Figure~\ref{pf-lem-replacement} shows that 
a nonself-chord oriented from $G_i$ to $G_j$ with sign $\varepsilon$ 
is replaced with a pair of nonself-chords 
one of which is oriented from $G_1$ to $G_i$ with sign $-\varepsilon$ 
and the other is from $G_1$ to $G_j$ with sign $\varepsilon$ 
by a combination of a $v\Delta_k^\circ$-move for some $k\in\{1,\dots,4\}$, a crossing change, 
and a Reidemeister move~II. 
Therefore we have the conclusion by Lemmas~\ref{lem-cc}(ii) and \ref{lem-delta}(ii). 
\end{proof}

\begin{figure}[htbp]
\centering
    \begin{overpic}[width=12cm]{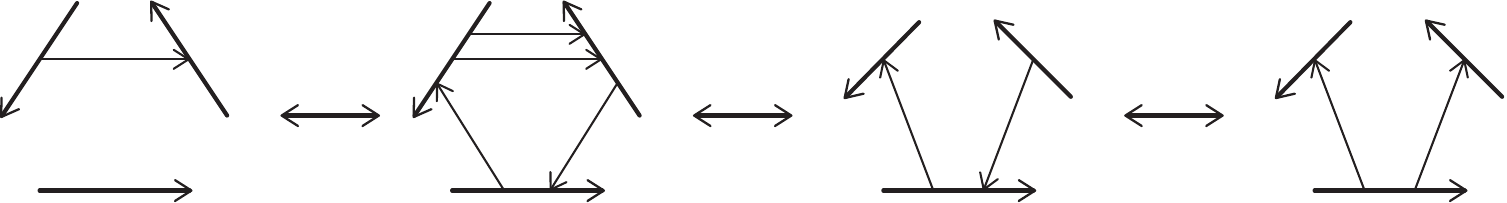}
      \put(66.5,27){$v\Delta_{k}^{\circ}$}
      \put(165,24.2){R}
      \put(262,24.2){cc}
      \put(23.5,34.3){$\e$}
      \put(-3,39){$G_{i}$}
      \put(42.5,39){$G_{j}$}
      \put(20,-9.9){$G_{1}$}
      \put(117.1,39.85){$\e$}
      \put(109.3,25.65){$-\e$}
      \put(92.2,10.35){$-\e$}
      \put(133,10.35){$-\e$}
      \put(90.5,39){$G_{i}$}
      \put(136.1,39){$G_{j}$}
      \put(114.6,-9.9){$G_{1}$}
      \put(188.3,14.56){$-\e$}
      \put(232.5,14.56){$-\e$}
      \put(189.8,39){$G_{i}$}
      \put(234,39){$G_{j}$}
      \put(213,-9.9){$G_{1}$}
      \put(286,14.56){$-\e$}
      \put(331,14.56){$\e$}
      \put(287.5,39){$G_{i}$}
      \put(332,39){$G_{j}$}
      \put(310,-9.9){$G_{1}$}
    \end{overpic}
  \vspace{1em}
  \caption{Proof of Lemma~\ref{lem-replacement}}
  \label{pf-lem-replacement}
\end{figure}

\begin{proof}[Proof of {\rm Proposition~\ref{prop-standard}}]
Let $G=\bigcup_{i=1}^n G_i$ be a Gauss diagram of $L$. 
Using forbidden (detour) moves and 
Reidemeister moves I, 
we can remove all the self-chords from $G$. 
By Lemmas~\ref{lem-delta-fd} and \ref{lem-delta-f}, 
we may assume that $G$ satisfies the condition (i)  
up to $v\Delta_i^\circ$-moves and Reidemeister moves. 

If there is a nonself-chord between $G_i$ and $G_j$ 
$(2\leq i\ne j\leq n)$, then we can replace it 
with a pair of chords between $G_1$ and $G_i$, 
and $G_1$ and $G_j$ by Lemma~\ref{lem-replacement}. 
Hence we may assume that $G$ satisfies 
the conditions (i) and (ii) up to $v\Delta_i^\circ$-moves and 
Reidemeister moves.

Finally, $G$ can be deformed into 
the one satisfying the conditions (i)--(vi) 
by forbidden (detour) moves, crossing changes, 
and Reidemeister moves II. 
Therefore we have the conclusion 
by Lemmas~\ref{lem-cc}(ii), \ref{lem-delta-fd}, and \ref{lem-delta-f}. 
\end{proof}

Let $L=\bigcup_{i=1}^n K_i$ be an oriented $n$-component virtual link, 
and $G=\bigcup_{i=1}^n G_i$ a Gauss diagram of $L$. 
The \textit{linking number} of an ordered pair 
$(K_i,K_j)$ is the sum of the signs of all the chords 
oriented from $G_i$ to $G_j$ $(1\leq i\ne j\leq n)$, 
which is an invariant of $L$ (cf.~\cite[Section 1.7]{GPV}). 
We denote it by ${\rm Lk}(K_i,K_j)$. 

For a chord $\gamma$ of $G$, 
it is convenient to introduce the signs of endpoints of $\gamma$ 
as follows. 
If the sign of $\gamma$ is $\varepsilon$, 
then we assign $-\varepsilon$ and $\varepsilon$ 
to the initial and terminal endpoints of $\gamma$, respectively. 
Then $-{\rm Lk}(K_i,K_j)$ 
is equal to the sum of the signs of all the endpoints 
of chords oriented from $G_i$ to $G_j$.

The \textit{$i$th intersection number} of $L$, denoted by $\lambda_i(L)$, 
is defined by 
\[\lambda_i(L)=\sum_{1\leq j\ne i\leq n} {\rm Lk}(K_j,K_i)
-\sum_{1\leq j\ne i\leq n}{\rm Lk}(K_i,K_j)\]
for $1\leq i\leq n$. 
Equivalently, $\lambda_i(L)$ is equal to 
the sum of the signs of all the endpoints of chords 
between $G_i$ and $G\setminus G_i$.

\begin{lemma}\label{lem-invariant}
If two oriented $n$-component virtual links 
$L$ and $L'$ are $v\Delta^\circ$-equivalent, 
then $\lambda_i(L)=\lambda_i(L')$ holds 
for any $1\leq i\leq n$. 
\end{lemma}

\begin{proof}
Every pair of three chords appeared in a $v\Delta_i^\circ$-move 
has two adjacent endpoints with opposite signs $\varepsilon$ and $-\varepsilon$. 
See Figure~\ref{pf-lem-invariant}. 
\end{proof}

\begin{figure}[htbp]
\centering
    \begin{overpic}[width=6cm]{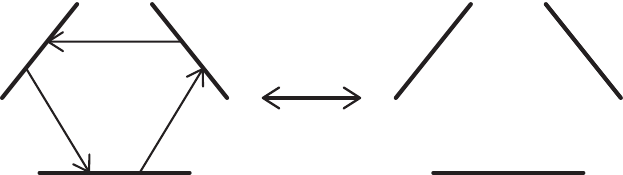}
      \put(77,29){$v\Delta_{i}^{\circ}$}
      \put(29,39.3){$\e$}
      \put(5.75,11){$\e$}
      \put(52,11){$\e$}
      \put(4.1,36.13){$\e$}
      \put(53,36.13){$-\e$}
      \put(-9.8,27.5){$-\e$}
      \put(60.8,27.5){$\e$}
      \put(21,-7.5){$\e$}
      \put(29,-7.5){$-\e$}
    \end{overpic}
  \caption{Proof of Lemma~\ref{lem-invariant}}
  \label{pf-lem-invariant}
\end{figure}

\begin{lemma}\label{lem-Minv}
Let $L=M(a_2,\dots,a_n)$ be the oriented $n$-component 
virtual link given in {\rm Proposition~\ref{prop-standard}}. 
Then we have 
$$\lambda_i(L)=
\begin{cases}
-(a_2+\dots+a_n) & \text{if }i=1, \\
a_i & \text{if }i=2,\dots,n. 
\end{cases}$$
\end{lemma}

\begin{proof}
This follows by definition immediately. 
\end{proof}

For example, 
the virtual link $L=M(2,0,4,-3)$ satisfies 
$$\lambda_1(L)=-3, \ \lambda_2(L)=2, \ 
\lambda_3(L)=0, \ \lambda_4(L)=4, 
\text{ and } \lambda_5(L)=-3.$$

\begin{proof}[Proof of {\rm (i)$\Leftrightarrow$(iii) in Theorem~\ref{thm12}}]
\underline{(i)$\Rightarrow$(iii).}
This follows from Lemma~\ref{lem-invariant}. 

\underline{(iii)$\Rightarrow$(i).} 
By Proposition~\ref{prop-standard}, 
$L$ and $L'$ are $v\Delta^\circ$-equivalent to 
$$M(a_2,\dots,a_n) \text{ and }
M(a_2',\dots,a_n')$$
for some $a_2,\dots,a_n$ and $a_2',\dots,a_n'\in{\Z}$, 
respectively. 
It follows from Lemmas~\ref{lem-invariant} and \ref{lem-Minv} 
that 
$$a_i=\lambda_i(L)=\lambda_i(L')=a_i'$$
for any $i=2,\dots,n$. 
Since $M(a_2,\dots,a_n)=M(a_2',\dots,a_n')$ holds, 
$L$ is $v\Delta^\circ$-equivalent to $L'$. 
\end{proof}

\section{$v\Delta^\wedge$-, $v\Delta^\circ$-, 
$v\sharp$-, and $vp$-unknotting numbers}\label{sec5}

In this section, 
we will consider the case of oriented virtual knots.

\begin{lemma}
For every $X\in\{v\Delta^\wedge, \ v\Delta^\circ, \ v\sharp, \ vp\}$,  
any two oriented virtual knots are $X$-equivalent to each other.
In particular, the $X$-move is an unknotting operation for oriented 
virtual knots. 
\end{lemma}

\begin{proof}
By Lemmas~\ref{lem-delta}(iii), \ref{lem-delta-sharp}, 
\ref{lem-delta-pass}, and \ref{lem-delta-f}, 
we have the following. 

\begin{center}
\begin{tabular}{ccccc}
$v\Delta_1^\wedge,\dots,v\Delta_4^\wedge$ 
&
$\Rightarrow$ 
&
$v\Delta_1^\circ,\dots,v\Delta_4^\circ$ 
&
$\Rightarrow$ 
&
$F_1,\dots,F_6$ \\
$\Updownarrow$ & & 
$\Updownarrow$
& & \\
$vp_1,\dots,vp_4$ & &
$v\sharp_1,v\sharp_2$
& & 
\end{tabular}
\end{center}

Therefore we see that 
if two oriented virtual knots are $F$-equivalent, 
then they are $X$-equivalent for every 
$X\in\{v\Delta^\wedge, \ v\Delta^\circ, \ v\sharp, \ vp\}$. 
Since any two oriented virtual knots are $F$-equivalent \cite{Kan,Nel}, 
they are $X$-equivalent. 
\end{proof}

For $X\in\{v\Delta, v\Delta^\wedge, \ v\Delta^\circ, \ v\sharp, \ vp\}$ and 
two oriented virtual knots $K$ and $K'$, 
we denote by ${\rm d}_X(K,K')$ the minimal number of 
$X$-moves which are required to deform a diagram of $K$ into that of $K'$. 
It is called the \textit{$X$-distance} between $K$ and $K'$. 
In particular, we denote ${\rm d}_X(K,O)$ by ${\rm u}_X(K)$, 
and call it the \textit{$X$-unknotting number} of $K$, 
where $O$ is the trivial knot.

We briefly review the $n$-writhe $J_n(K)$ 
and the odd writhe $J(K)$ of an oriented virtual knot $K$, 
which are invariants of $K$ (cf. \cite{ST}). 
Let $G$ be a Gauss diagram of $K$, 
and $\gamma$ a chord of $G$.  
The endpoints of $\gamma$ divide the underlying oriented circle of $G$ 
into two arcs. 
Let $\alpha$ be the one of the two arcs 
oriented from the initial endpoint of $\gamma$ to the terminal. 
The \textit{index} of $\gamma$ 
is the sum of the signs of all the endpoints of chords on $\alpha$, 
and denoted by ${\rm Ind}(\gamma)$. 
For $n\ne 0$, the sum of the signs of all the chords with index $n$ 
is an invariant of $K$. 
It is called the \textit{$n$-writhe} of $K$, 
and denoted by $J_{n}(K)$. 
Furthermore the \textit{odd writhe} of $K$ is defined to be 
$J(K)=\sum_{n{\rm :odd}}J_n(K)$.


\begin{proposition}\label{prop-lower-odd}
For two oriented virtual knots $K$ and $K'$, 
we have the following. 
\begin{enumerate}
\item
${\rm d}_X(K,K')\geq 
\frac{1}{2}|J(K)-J(K')|$ for 
$X\in\{v\Delta^\wedge, v\Delta^\circ, vp\}$. 
\item
${\rm d}_{v\sharp}(K,K')\geq 
\frac{1}{4}|J(K)-J(K')|$. 
\end{enumerate}
In particular, we have 
$${\rm u}_X(K)\geq \tfrac{1}{2}|J(K)| 
\text{ for }X\in\{v\Delta^\wedge, v\Delta^\circ, vp\} 
\text{ and }
{\rm u}_{v\sharp}(K)\geq \tfrac{1}{4}|J(K)|.$$
\end{proposition}

\begin{proof}
(i) 
For $X\in\{v\Delta^\wedge, v\Delta^\circ\}$,
if $K$ and $K'$ are $X$-equivalent,
then they are $v\Delta$-equivalent by definition.
Thus we have ${\rm d}_{X}(K,K')\geq  {\rm d}_{v\Delta}(K,K')$.
Since it follows from \cite[Proposition~2.6]{NNSW} that 
${\rm d}_{v\Delta}(K,K')\geq \frac{1}{2}|J(K)-J(K')|$ holds, 
we have the inequality.

For $X=vp$, a $vp$-move contains two positive and two negative
real crossings. Therefore 
a single $vp$-move changes the odd writhe by at most two.

\noindent(ii) Since a $v\sharp$-move contains four real crossings,
a single $v\sharp$-move changes the odd writhe by at most four.
\end{proof}

\begin{proposition}\label{prop-lower-writhe}
For two oriented  virtual knots $K$ and $K'$, 
we have the following. 
\begin{enumerate}
\item
${\rm d}_{v\Delta^\circ}(K,K')\geq 
\frac{1}{3}\sum_{n\ne 0} |J_n(K)-J_n(K')|.$
\item
${\rm d}_{vp}(K,K')\geq 
\frac{1}{4}\sum_{n\ne 0}|J_n(K)-J_n(K')|.$
\end{enumerate}
In particular, we have 
${\rm u}_{v\Delta^\circ}(K)\geq \tfrac{1}{3}\sum_{n\ne 0}|J_n(K)|
\text{ and }
{\rm u}_{vp}(K)\geq \frac{1}{4}\sum_{n\ne 0}|J_n(K)|.$
\end{proposition}

\begin{proof}
(i) A $v\Delta^\circ$-move does not change the index 
of any chord except for the three chords involved in the move. 
See Figure~\ref{pf-lem-invariant} again.
Therefore if $K$ and $K'$ are related by a single $v\Delta^\circ$-move,
then we have $\sum_{n\ne 0} |J_n(K)-J_n(K')|\leq 3$.

\noindent(ii) A $vp$-move does not change the index 
of any chord except for the four chords involved in the move. 
See Figure~\ref{fig:delta-pass-n-writhe} as an example.
Therefore if $K$ and $K'$ are related by a single $vp$-move,
then we have $\sum_{n\ne 0} |J_n(K)-J_n(K')|\leq 4$.
\end{proof}

\begin{figure}[htbp]
\centering
    \begin{overpic}[width=6cm]{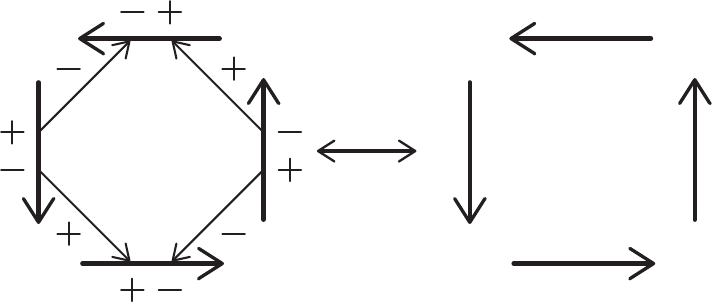}
      \put(81,43){$vp_{1}$}
    \end{overpic}
  \caption{An example of a $vp$-move}
  \label{fig:delta-pass-n-writhe}
\end{figure}

Theorem~\ref{thm-infinite} is decomposed into 
Theorems~\ref{thm-vdw}--\ref{thm-vp} as follows. 

\begin{theorem}\label{thm-vdw}
For any positive integer $m$, 
there are infinitely many oriented virtual knots $K$ 
with ${\rm u}_{v\Delta^\wedge}(K)=m$. 
\end{theorem}

\begin{proof}
For an integer $s\geq 2$, 
we consider a long virtual knot $T_s$ 
presented by a diagram 
as shown in the left of Figure~\ref{pf-thm-infinite-vdw}, 
where the vertical twists consist of 
$2s$ positive real crossings and $2s-1$ virtual crossings. 
By taking the closure of the product of $m$ copies of $T_s$, 
we obtain an oriented  virtual knot $K_s(m)$ as in the right of the figure. 

\begin{figure}[htbp]
\centering
    \begin{overpic}[width=11cm]{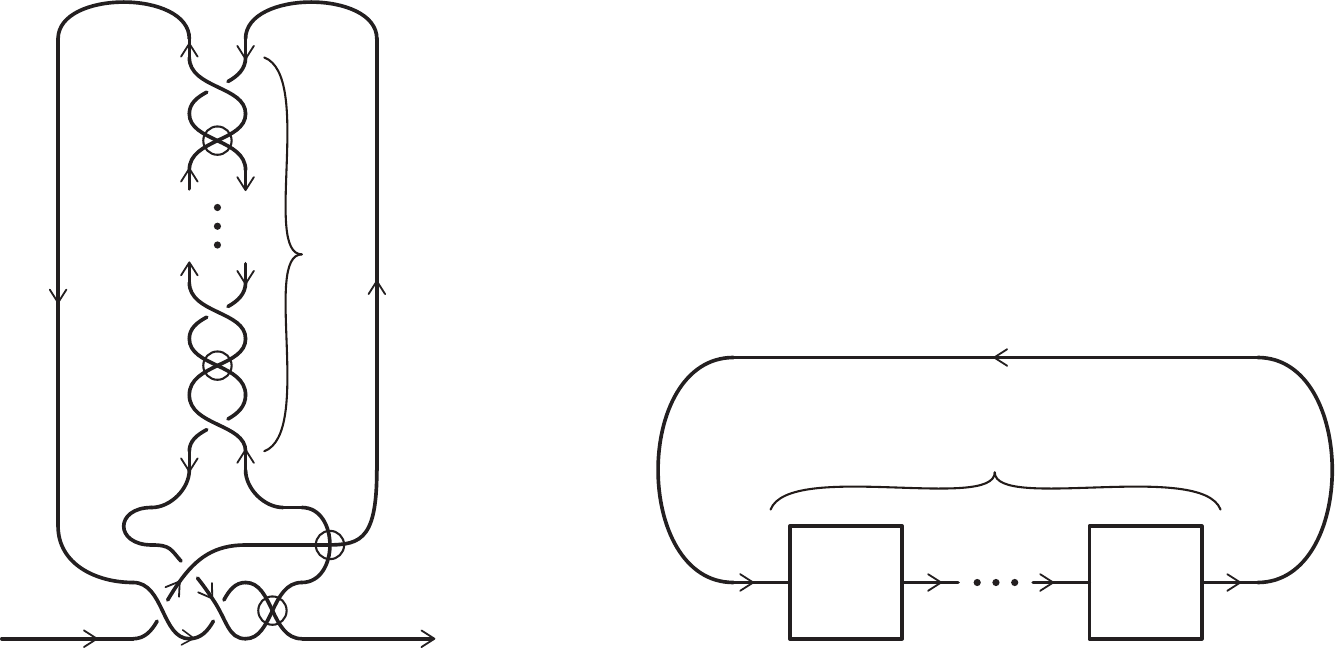}
      \put(74,127){\rotatebox{-90}{$4s-1$ crossings}}
      \put(42.5,6.5){$*$}
      \put(48,-13){$T_{s}$}
      \put(201,49){$m$ copies of $T_{s}$}
      \put(194.8,12){$T_{s}$}
      \put(265,12){$T_{s}$}
      \put(220,-13){$K_s(m)$}
    \end{overpic}
  \vspace{1em}
  \caption{Diagrams of $T_{s}$ and $K_s(m)$}
  \label{pf-thm-infinite-vdw}
\end{figure}
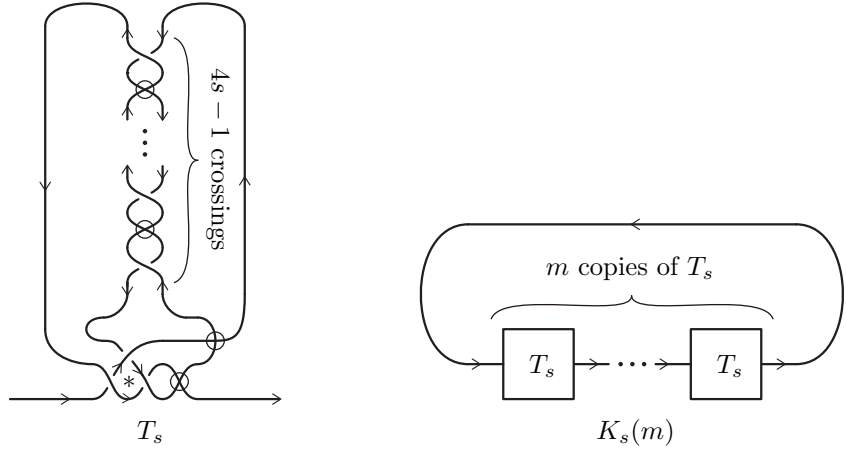

As shown in the proof of~\cite[Theorem~2.9]{NNSW}, 
the set $\{K_s(m)\mid s\geq 2\}$ gives an infinite family of 
oriented virtual knots with ${\rm u}_{v\Delta}(K_s(m))=m$. 
Since the long knot diagram of $T_s$ can be unknotted 
by a $v\Delta^\wedge$-move for the three real crossings 
around the region with the mark $*$, 
we have ${\rm u}_{v\Delta^\wedge}(K_s(m))=m$. 
\end{proof}

\begin{theorem}\label{thm-vdc}
For any positive integer $m$, 
there are infinitely many oriented virtual knots $K$ 
with ${\rm u}_{v\Delta^\circ}(K)=m$. 
\end{theorem}

\begin{proof}
For an integer $s\geq 1$, let $T_s$ be a long virtual knot presented by 
a diagram as shown in the top of Figure~\ref{pf-thm-infinite-vdc}. 
Then its Gauss diagram is shown 
in the bottom of the figure, 
and has $4s+3$ chords $a_i$ $(i=1,2,\ldots,2s)$, 
$b_j$ $(j=1,2,3)$, and $c_k$ $(k=1,2,\ldots,2s)$ 
with signs 
\[\e(a_i)=+1, \ \e(b_j)=-1, \text{ and }\e(c_k)=+1,\]
where $\e(\gamma)$ denotes the sign of a chord $\gamma$. 
Let $K_s(m)$ be an oriented virtual knot as the closure of the product 
of $m$ copies of $T_s$.

\begin{figure}[htbp]
\centering
    \begin{overpic}[width=11cm]{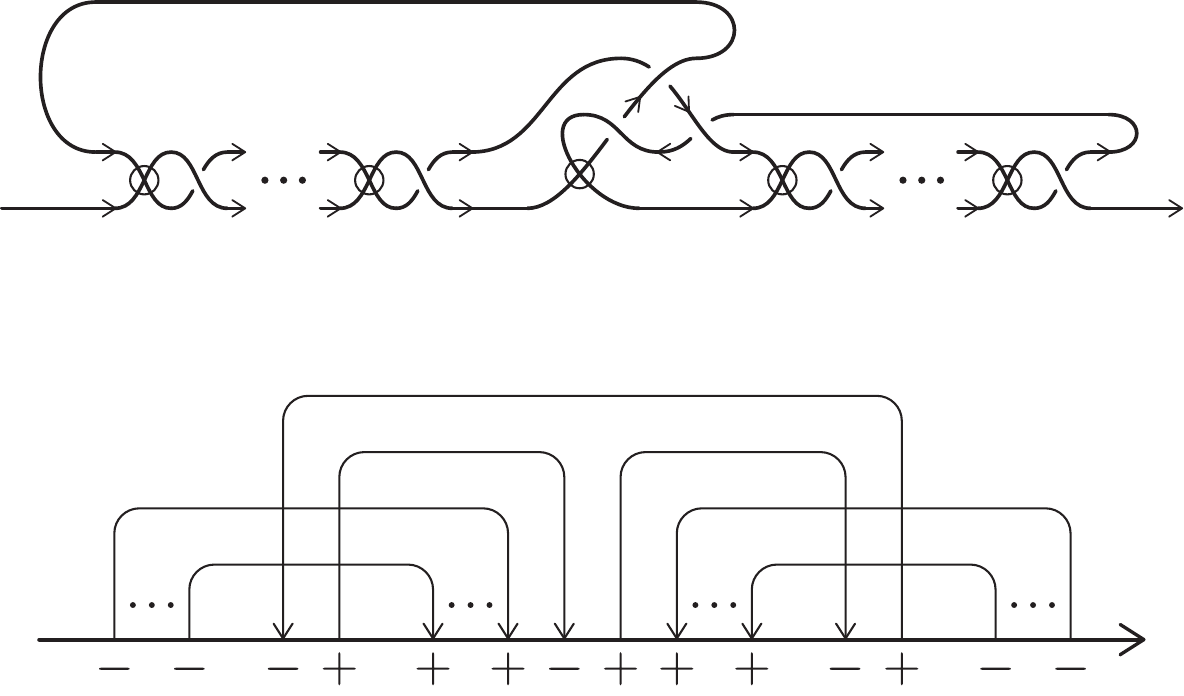}
      \put(49,122){$a_{1}$}
      \put(107,122){$a_{2s}$}
      \put(219,122){$c_{1}$}
      \put(277,122){$c_{2s}$}
      \put(161,138){$b_{1}$}
      \put(171,172){$b_{2}$}
      \put(183,138){$b_{3}$}
      \put(18,22){$a_{1}$}
      \put(53,22){$a_{2s}$}
      \put(168,22){$c_{1}$}
      \put(202,22){$c_{2s}$}
      \put(115,69){$b_{2}$}
      \put(190,69){$b_{3}$}
      \put(154,85){$b_{1}$}
    \end{overpic}
  \caption{A diagram of $T_{s}$ and its Gauss diagram}
  \label{pf-thm-infinite-vdc}
\end{figure}
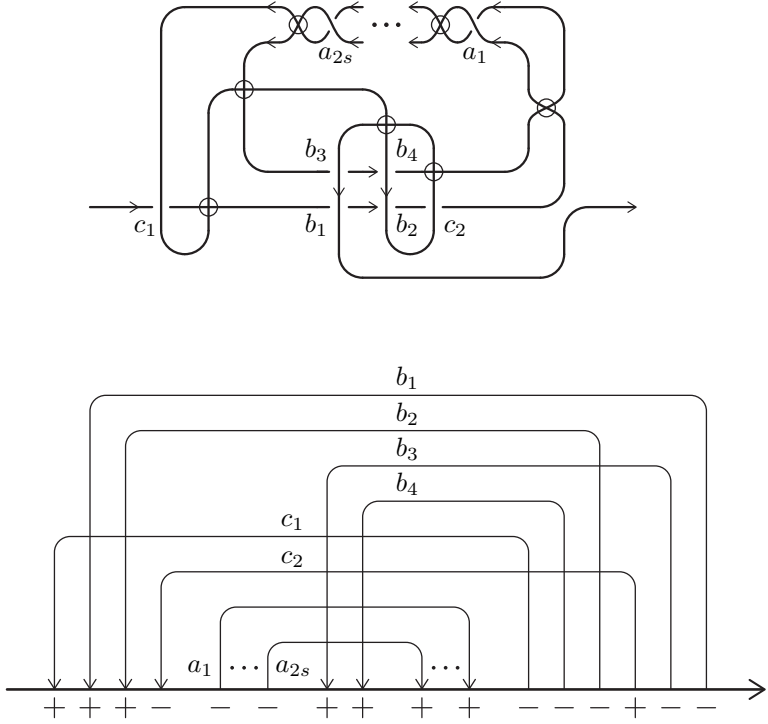

We can apply a $v\Delta^\circ$-move to the three real crossings $b_1$,  $b_2$, 
and $b_3$ on the long knot diagram 
so that $T_s$ becomes unknotted. 
Thus we have ${\rm u}_{v\Delta^\circ}(K_s(m))\leq m$.

To prove ${\rm u}_{v\Delta^\circ}(K_s(m))\geq m$, 
we will calculate the $n$-writhe of $K_s(m)$ as follows. 
Since we have 
$${\rm Ind}(a_i)={\rm Ind}(c_k)=0,\ 
{\rm Ind}(b_1)=-4s,\text{ and } {\rm Ind}(b_2)={\rm Ind}(b_3)=2s,$$
it holds that 
$$J_n(K_s(m))=
\begin{cases}
-2m & \text{if }n=2s,\\
-m & \text{if }n=-4s, \\
0 & \text{otherwise}.
\end{cases}$$
By Proposition~\ref{prop-lower-writhe}(i), 
we have 
${\rm u}_{v\Delta^\circ}(K_s(m))\geq \frac{1}{3}(2m+m)=m$, 
and hence ${\rm u}_{v\Delta^\circ}(K_s(m))=m$. 

Furthermore for any $s> s'$, since 
$$J_{2s}(K_s(m))=-2m\neq 0=J_{2s}(K_{s'}(m))$$ holds, 
we have $K_s(m)\neq K_{s'}(m)$.
\end{proof}

\begin{theorem}\label{thm-vs}
For any positive integer $m$, 
there are infinitely many oriented virtual knots $K$ 
with ${\rm u}_{v\sharp}(K)=m$. 
\end{theorem}

\begin{proof}
For an integer $s\geq 3$, let $T_s$ be  
a long virtual knot presented by 
a diagram as shown in the top of Figure~\ref{pf-thm-infinite-vs}. 
Then its Gauss diagram is shown 
in the bottom of the figure, 
and has $2s+6$ chords $a_i$ $(i=1,2,\ldots,2s)$, 
$b_j$ $(j=1,2,3,4)$, and $c_k$ $(k=1,2)$ 
with signs 
\[\e(a_i)=\e(b_j)=\e(c_1)=+1, \text{ and }\e(c_2)=-1.\] 
Let $K_s(m)$ be an oriented virtual knot as the closure of the product 
of $m$ copies of $T_s$.

\begin{figure}[htbp]
\centering
    \begin{overpic}[width=10cm]{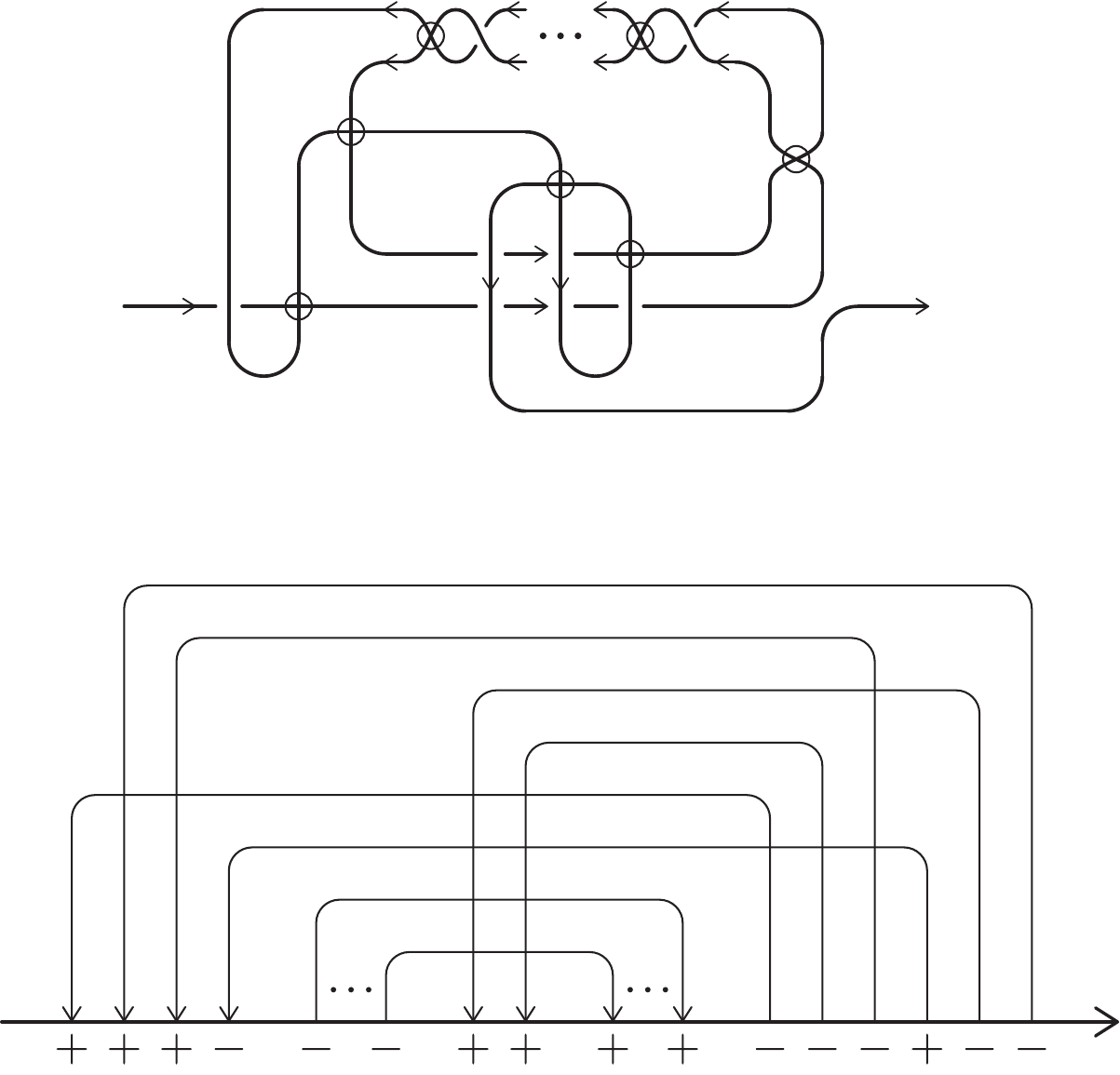}
      \put(48,186){$c_{1}$}
      \put(112,186){$b_{1}$}
      \put(146,186){$b_{2}$}
      \put(164,186){$c_{2}$}
      \put(171,251){$a_{1}$}
      \put(117,251){$a_{2s}$}
      \put(112,214){$b_{3}$}
      \put(146,214){$b_{4}$}
      \put(103,75){$c_{1}$}
      \put(103,61){$c_{2}$}
      \put(146,128){$b_{1}$}
      \put(146,114.5){$b_{2}$}
      \put(146,101.5){$b_{3}$}
      \put(146,88){$b_{4}$}
      \put(68,20){$a_{1}$}
      \put(101,20){$a_{2s}$}
    \end{overpic}
  \caption{A diagram of $T_{s}$ and its Gauss diagram}
  \label{pf-thm-infinite-vs}
\end{figure}

We can apply a $v\sharp$-move to $b_1$, $b_2$, $b_3$,  
and $b_4$ so that $T_s$ becomes unknotted. 
Thus we have ${\rm u}_{v\sharp}(K_s(m))\leq m$.

On the other hand, since we have 
\begin{center}
\begin{tabular}{l}
${\rm Ind}(a_i)=2$, \  
${\rm Ind}(b_1)={\rm Ind}(b_2)={\rm Ind}(c_2)=1$, 
\medskip \\
${\rm Ind}(b_3)={\rm Ind}(b_4)=-2s+1$, and ${\rm Ind}(c_1)=-3$, 
\end{tabular}
\end{center}
it holds that 
$$J_n(K_s(m))=
\begin{cases}
2ms & \text{if }n=2,\\
m & \text{if }n=1, -3, \\
2m & \text{if }n=-2s+1, \\
0 & \text{otherwise}.
\end{cases}$$
This induces $J(K_s(m))=m+m+2m=4m$. 
Therefore we have ${\rm u}_{v\sharp}(K_s(m))\geq m$ 
by Proposition~\ref{prop-lower-odd}(ii), 
and hence ${\rm u}_{v\sharp}(K_s(m))= m$.

Furthermore for any $s> s'$, since 
$$J_{-2s+1}(K_s(m))=2m\neq 0=J_{-2s+1}(K_{s'}(m))$$ holds, 
we have $K_s(m)\neq K_{s'}(m)$.
\end{proof}

\begin{theorem}\label{thm-vp}
For any positive integer $m$, 
there are infinitely many oriented virtual knots $K$ 
with ${\rm u}_{vp}(K)=m$. 
\end{theorem}

\begin{proof}
For an integer $s\geq 1$, 
let $T_s$ be a long virtual knot presented by a diagram 
as shown in the top of Figure~\ref{pf-thm-infinite-vp}. 
Then its Gauss diagram is shown in the bottom of the figure, 
and has $4s+8$ chords 
$a_i$ $(i=1,2,\dots,2s)$, 
$b_j$ $(j=1,2,3,4)$, 
$c_k$ $(k=1,2)$, and 
$d_{\ell}$ $(\ell=1,2,\dots,2s+2)$ 
with signs 
\[\e(a_i)=\e(b_2)=\e(b_4)=\e(c_k)=\e(d_{\ell})=+1 \text{ and }
\e(b_1)=\e(b_3)=-1.\]
Let $K_s(m)$ be an oriented virtual knot as the closure of 
the product of $m$ copies of $T_s$. 

\begin{figure}[htbp]
\centering
    \begin{overpic}[width=12cm]{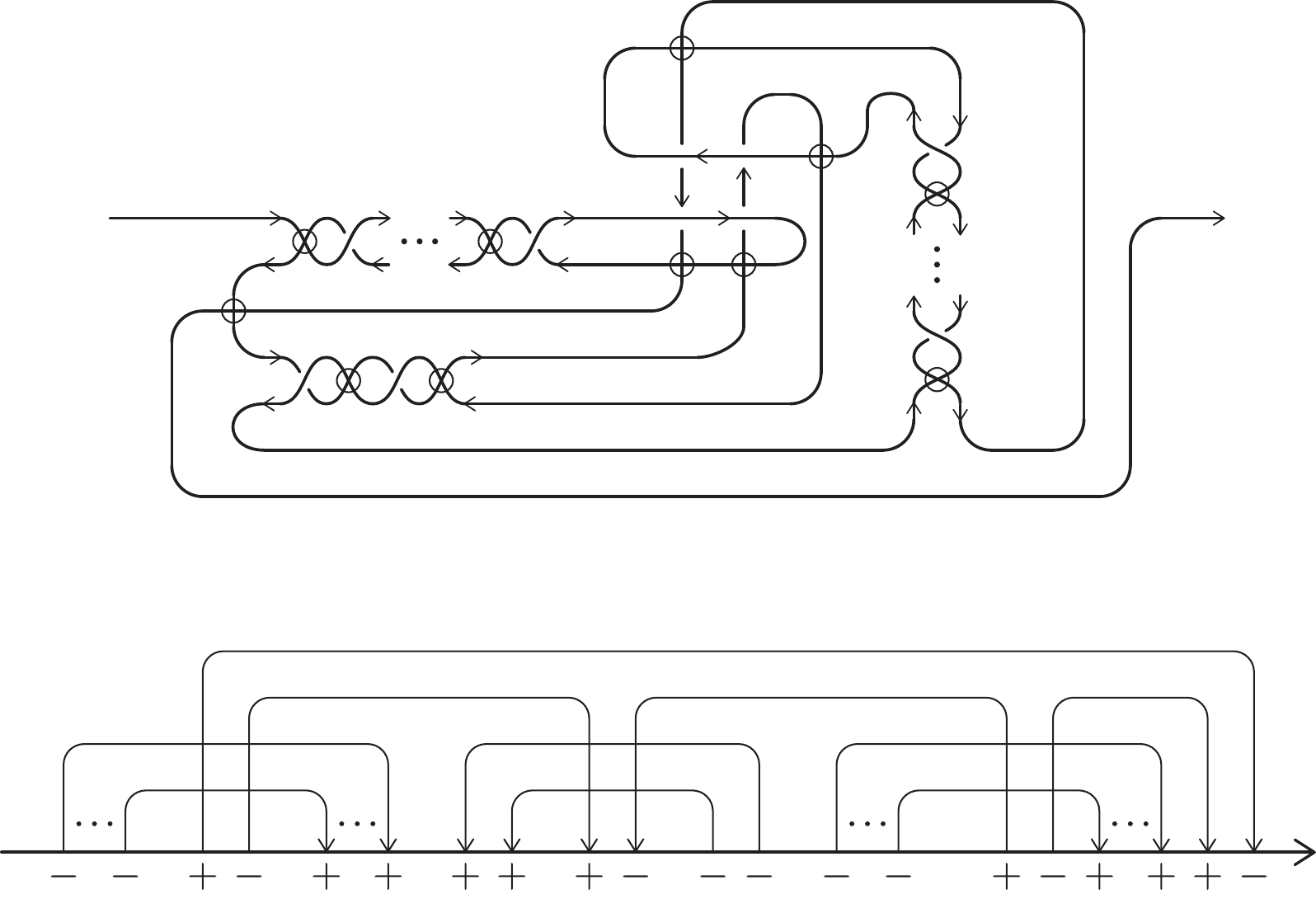}
      \put(86,181.5){$a_{1}$}
      \put(132,181.5){$a_{2s}$}
      \put(165,168){$b_{1}$}
      \put(196,168){$b_{2}$}
      \put(196,196){$b_{3}$}
      \put(165,196){$b_{4}$}
      \put(75,144.5){$c_{1}$}
      \put(99,144.5){$c_{2}$}
      \put(253,192){$d_{2s+2}$} 
      \put(253,144){$d_{1}$}
      \put(5,17){$a_{1}$}
      \put(35,17){$a_{2s}$}
      \put(109,17){$c_{1}$}
      \put(136,17){$c_{2}$}
      \put(205,17){$d_{1}$}
      \put(235,17){$d_{2s+2}$}
      \put(185,69){$b_{1}$}
      \put(105,55.5){$b_{2}$}
      \put(205,55.5){$b_{3}$}
      \put(289,55.5){$b_{4}$}
    \end{overpic}
  \caption{A diagram of $T_{s}$ and its Gauss diagram}
  \label{pf-thm-infinite-vp}
\end{figure}
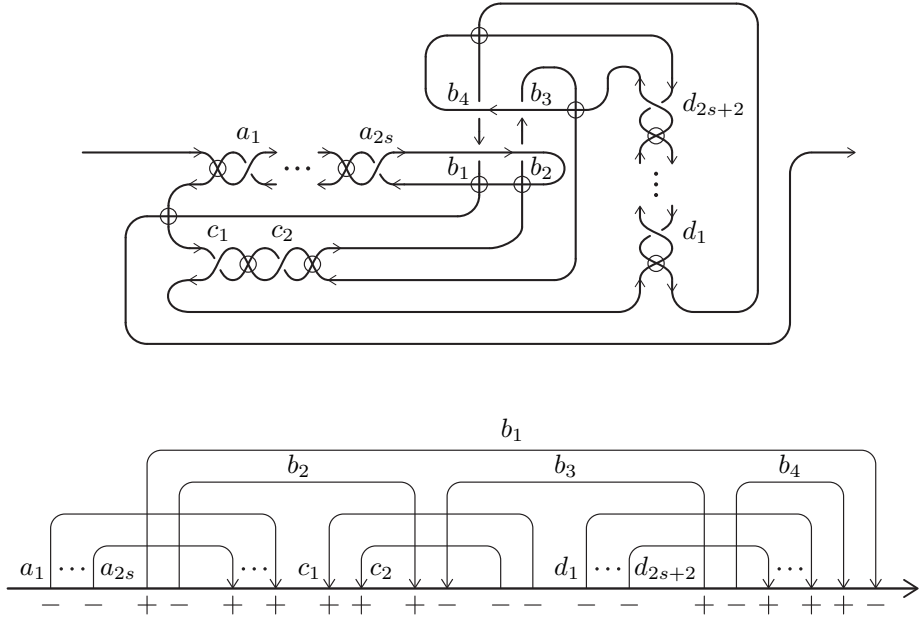

We can apply a $vp$-move to 
$b_1$, $b_2$, $b_3$, and $b_4$ 
so that $T_s$ becomes unknotted. 
Thus we have 
${\rm u}_{vp}(K_s(m))\leq m$. 

On the other hand, since we have 
\[\begin{split}
&{\rm Ind}(a_i)={\rm Ind}(c_k)={\rm Ind}(d_\ell)=0, \ 
{\rm Ind}(b_1)=2s, \\
&{\rm Ind}(b_2)={\rm Ind}(b_4)=2s+2, \text{ and }
{\rm Ind}(b_3)=2s+4, 
\end{split}\]
it holds that 
$$J_n(K_s(m))=
\begin{cases}
-m & \text{if }n=2s, 2s+4,\\
2m & \text{if }n=2s+2, \\
0 & \text{otherwise}.
\end{cases}$$
Therefore we have ${\rm u}_{vp}(K_s(m))\geq 
\frac{1}{4}(m+m+2m)=m$ 
by Proposition~\ref{prop-lower-writhe}(ii), 
and hence ${\rm u}_{vp}(K_s(m))=m$.

Furthermore for any $s>s'$, since 
$$J_{2s+4}(K_s(m))=-m\neq 0=J_{2s+4}(K_{s'}(m))$$ holds, 
we have $K_s(m)\neq K_{s'}(m)$. 
\end{proof}

We remark that the oriented virtual knots $K=K_s(m)$ 
constructed in the proof of Theorem~\ref{thm-vdc} satisfy 
\[{\rm u}_{v\Delta^\circ}(K)> \frac{1}{2}|J(K)|
\text{ and }
{\rm u}_{v\Delta^\circ}(K)= \frac{1}{3}
\sum_{n\ne 0}|J_n(K)|=m.\]
In fact, we have $J(K)=0$. 
Generally the two lower bounds for ${\rm u}_{v\Delta^\circ}(K)$ 
given in Propositions~\ref{prop-lower-odd}(i) 
and \ref{prop-lower-writhe}(i) are independent 
in the following sense. 

\begin{proposition}
For any positive integer $m$, 
there are infinitely many oriented virtual knots $K$ 
with 
\[{\rm u}_{v\Delta^\circ}(K)= \frac{1}{2}|J(K)|=m
\text{ and }
{\rm u}_{v\Delta^\circ}(K)> \frac{1}{3}
\sum_{n\ne 0}|J_n(K)|.\]
\end{proposition}

\begin{proof}
For an integer $s\geq 1$, 
let $T_s$ be a long virtual knot presented by 
a diagram as shown in the top of Figure~\ref{ex-infinite-vdc}. 
Then its Gauss diagram is shown in the bottom of the figure, 
and has $2s+4$ positive chords $a_i$ 
$(i=1,2,\ldots,2s-1)$, 
$b_j$ $(j=1,2,3)$, and 
$c_k$ $(k=1,2)$. 
Let $K_s(m)$ be an oriented virtual knot as the closure of the product 
of $m$ copies of $T_s$.

\begin{figure}[htbp]
\centering
    \begin{overpic}[width=10cm]{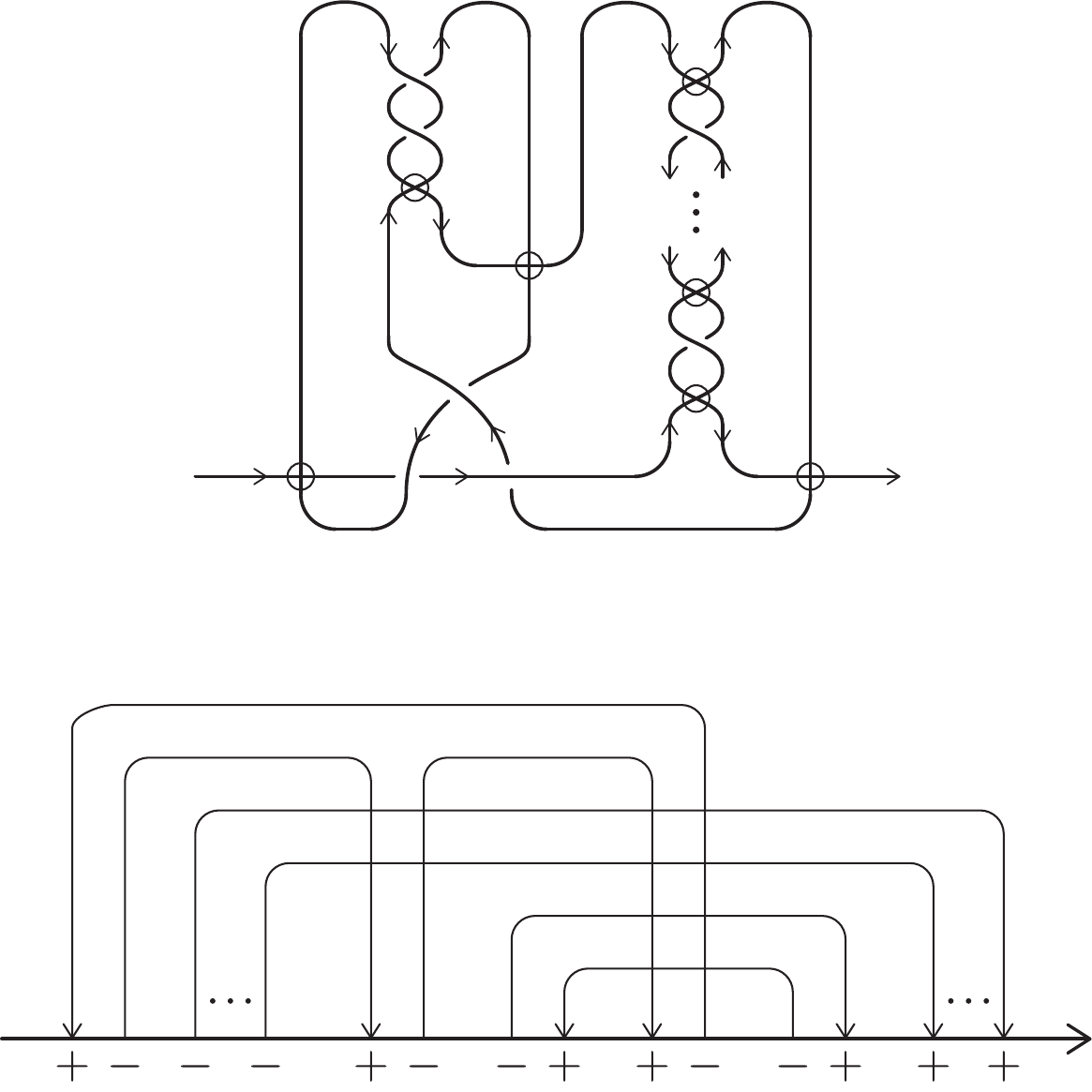}
      \put(96,166){$b_{1}$}
      \put(135.5,166){$b_{2}$}
      \put(115.5,190){$b_{3}$}
      \put(190,193){$a_{1}$}
      \put(189,248){{\small $a_{2s-1}$}}
      \put(117,247){$c_{1}$}
      \put(117,262){$c_{2}$}
      \put(100,105){$b_{1}$}
      \put(59,91){$b_{2}$}
      \put(135,91){$b_{3}$}
      \put(39,22){$a_{1}$}
      \put(71.5,22){$a_{2s-1}$}
      \put(122,22){$c_{1}$}
      \put(136.5,22){$c_{2}$}
    \end{overpic}
  \caption{A diagram of $T_{s}$ and its Gauss diagram}
  \label{ex-infinite-vdc}
\end{figure}
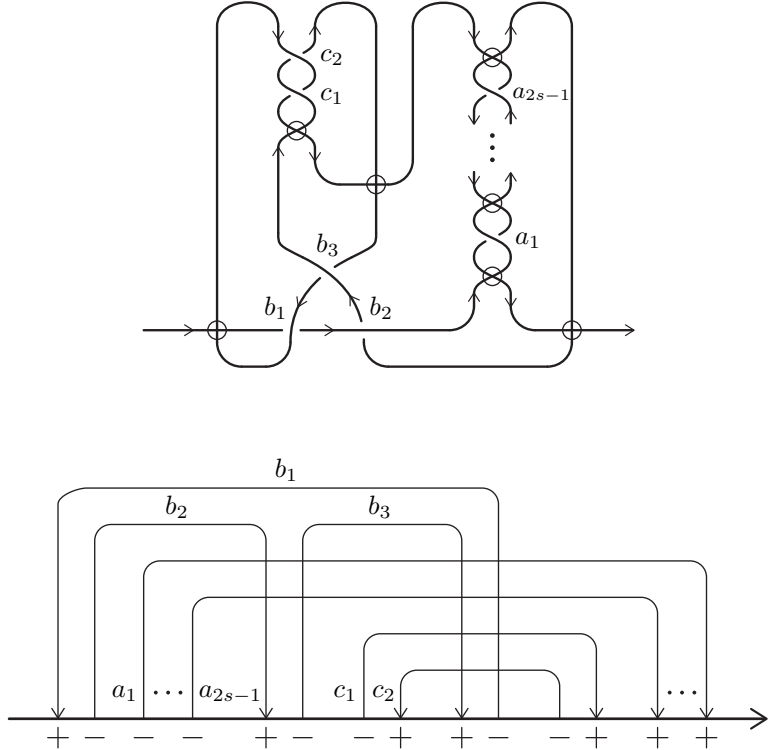

We can apply a $v\Delta^\circ$-move to $b_1$,  $b_2$, 
and $b_3$ so that $T_s$ becomes unknotted. 
Thus we have ${\rm u}_{v\Delta^\circ}(K_s(m))\leq m$.

On the other hand, since we have 
$${\rm Ind}(a_i)={\rm Ind}(c_ k)={\rm Ind}(b_3)=0,\ 
{\rm Ind}(b_1)=2s-1,\ \text{and } {\rm Ind}(b_2)=-2s+1,$$
it holds that 
$$
J_{n}(K_s(m))=
\begin{cases}
m & \text{if }n=2s-1,-2s+1, \\ 
0 & \text{otherwise}. 
\end{cases}$$ 
This induces $J(K_s(m))=m+m=2m$. 
By Propositions~\ref{prop-lower-odd}(i) and~\ref{prop-lower-writhe}(i), 
we have 
\[
\begin{split}
&{\rm u}_{v\Delta^\circ}(K_s(m))= \frac{1}{2}|J(K_s(m))|=m
\text{ and }\\
&{\rm u}_{v\Delta^\circ}(K_s(m))> \frac{1}{3}
\sum_{n\ne 0}|J_n(K_s(m))|=\dfrac{2}{3}m.
\end{split}\]

Furthermore for any $s> s'$, since 
$$J_{2s-1}(K_s(m))=m\neq 0=J_{2s-1}(K_{s'}(m))$$
holds, 
we have $K_s(m)\neq K_{s'}(m)$.
\end{proof}

Similarly to the case above, 
the oriented virtual knots $K=K_s(m)$ 
constructed in the proof of Theorem~\ref{thm-vp} satisfy 
\[{\rm u}_{vp}(K)> \frac{1}{2}|J(K)|
\text{ and }
{\rm u}_{vp}(K)= \frac{1}{4}
\sum_{n\ne 0}|J_n(K)|=m.\]
In fact, we have $J(K)=0$. 
Generally the two lower bounds for ${\rm u}_{vp}(K)$ 
given in Propositions~\ref{prop-lower-odd}(i) 
and \ref{prop-lower-writhe}(ii) are independent 
in the following sense. 

\begin{proposition}
For any positive integer $m$, 
there are infinitely many oriented virtual knots $K$ 
with 
\[{\rm u}_{vp}(K)= \frac{1}{2}|J(K)|=m
\text{ and }
{\rm u}_{vp}(K)> \frac{1}{4}
\sum_{n\ne 0}|J_n(K)|.\]
\end{proposition}

\begin{proof}
For an integer $s\geq 2$, let $T_s$ be a long virtual knot presented by 
a diagram as shown in the top of Figure~\ref{ex-infinite-vp}. 
Then its Gauss diagram is shown in the bottom of the figure, 
and has $2s+5$ chords $a_i$ $(i=1,2,\ldots,2s)$, 
$b_j$ $(j=1,2,3,4)$, and 
$c$ with signs 
$$\e(a_i)=\e(b_1)=\e(b_3)=\e(c)=+1 \text{ and }
\e(b_2)=\e(b_4)=-1.$$
Let $K_s(m)$ be an oriented virtual knot as the closure of the product 
of $m$ copies of $T_s$.

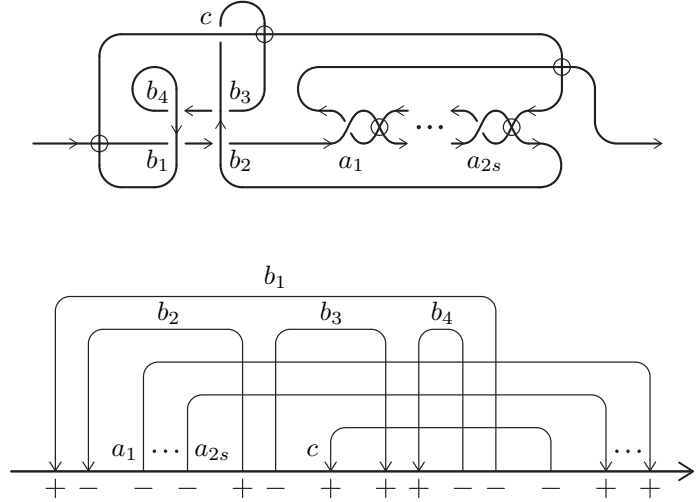
\begin{figure}[htbp]
\centering
    \begin{overpic}[width=9cm]{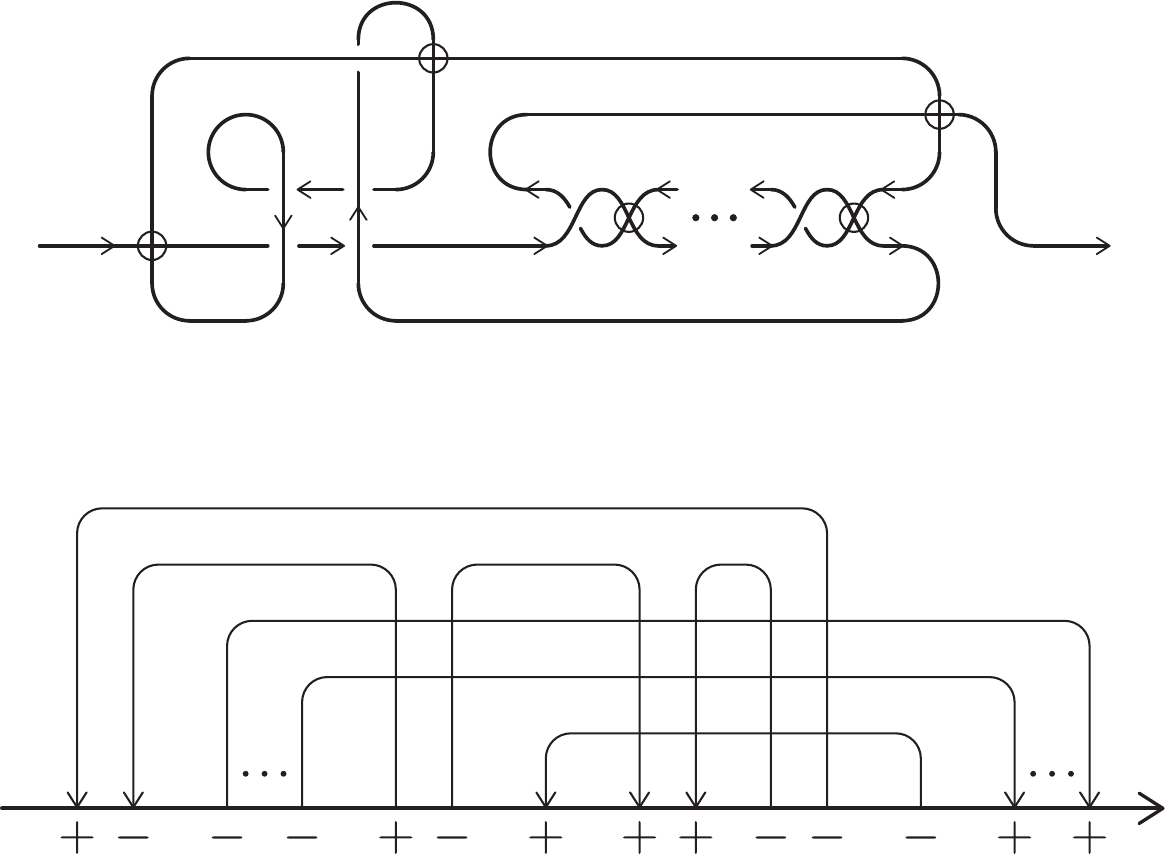}
      \put(51,127){$b_{1}$}
      \put(82,127){$b_{2}$}
      \put(82,153){$b_{3}$}
      \put(51,153){$b_{4}$}
      \put(71,181){$c$}
      \put(123,127){$a_{1}$}
      \put(171,127){$a_{2s}$}
      \put(95,83){$b_{1}$}
      \put(55,69.5){$b_{2}$}
      \put(116,69.5){$b_{3}$}
      \put(157.5,69.5){$b_{4}$}
      \put(38,18){$a_{1}$}
      \put(69,18){$a_{2s}$}
      \put(111,18){$c$}
    \end{overpic}
  \caption{A diagram of $T_{s}$ and its Gauss diagram}
  \label{ex-infinite-vp}
\end{figure}

We can apply a $vp$-move to $b_1$,  $b_2$, $b_3$, 
and $b_4$ so that $T_s$ becomes unknotted. 
Thus we have ${\rm u}_{vp}(K_s(m))\leq m$.

On the other hand, since we have 
$${\rm Ind}(a_i)={\rm Ind}(c)={\rm Ind}(b_4)=0,\ 
{\rm Ind}(b_1)=2s-1,\ {\rm Ind}(b_2)=2s, \text{and } {\rm Ind}(b_3)=1,$$
it holds that 
$$J_{n}(K_s(m))=
\begin{cases}
m & \text{if }n=1,2s-1,\\ 
-m & \text{if }n=2s,\\ 
0 & \text{otherwise}.
\end{cases}$$
This induces $J(K_s(m))=m+m=2m$. 
By Propositions~\ref{prop-lower-odd}(i) and \ref{prop-lower-writhe}(ii), 
we have 
\[
\begin{split}
&{\rm u}_{vp}(K_s(m))= \frac{1}{2}|J(K_s(m))|=m
\text{ and }\\
&{\rm u}_{vp}(K_s(m))> \frac{1}{4}
\sum_{n\ne 0}|J_n(K_s(m))|=\frac{3}{4}m.
\end{split}\]

Furthermore for any $s\neq s'$, since 
$$J_{2s}(K_s(m))=-m\neq 0=J_{2s}(K_{s'}(m))$$ holds, 
we have $K_s(m)\neq K_{s'}(m)$.
\end{proof}


\end{document}